\documentclass[11pt,reqno]{amsart}

\usepackage{amsmath, mathtools}
\usepackage{amsthm}
\usepackage{graphicx}
\usepackage{upgreek}
\usepackage{hyperref}
\usepackage{mathrsfs}
\usepackage{xcolor}
\usepackage{bm,bbm}
\usepackage{amsfonts}
\usepackage{amssymb}
\usepackage{csquotes}
\usepackage{stmaryrd}
\usepackage[normalem]{ulem}

\usepackage{cancel, soul}
\usepackage{enumerate}
\usepackage{comment}

\usepackage{soul}
\usepackage[margin=1.1in]{geometry}
\numberwithin{equation}{section}
\newtheorem{theorem}{Theorem}[section]
\newtheorem{lemma}{Lemma}[section]
\newtheorem{assumption}{Assumption}[section]
\newtheorem{proposition}{Proposition}[section]

\newtheorem{definition}{Definition}[section]
\newtheorem{remark}{Remark}[section]

\newcommand{\E}{\mathbb{E}}
\newcommand{\R}{\mathbb{R}}

\newcommand{\N}{\mathbb{N}}

\newcommand{\sig}{\sigma}

\newcommand \bet {\beta}

\newcommand{\ep}{\varepsilon}
\newcommand{\vr}{\varrho}

\newcommand{\kap}{\kappa}

\newcommand{\del}{\delta}

\newcommand{\Gam}{\mathnormal{\Gamma}}

\newcommand{\Q}{{\mathbb Q}}

\newcommand \F {\mathbb{F}}

\newcommand{\PP}{{\mathbb P}}

\newcommand{\calA}{{\mathcal A}}

\newcommand{\calC}{{\mathcal C}}

\newcommand{\calF}{{\mathcal F}}

\newcommand{\calQ}{{\mathcal Q}}

\newcommand{\lt}{\left}
\newcommand{\rt}{\right}

\newcommand\iy{\infty}

\DeclareMathOperator*{\argmin}{arg\,min}

\newcommand{\one}{\mathbbm{1}}
\allowdisplaybreaks

\newcommand{\sredm}[1]{\ifmmode\text{\xout{\ensuremath{\displaystyle \textcolor{red}{#1}}}}\else\sout{\textcolor{red}{#1}}\fi}

\newcommand{\magenta}[1]{{\color{magenta}{#1}}}

\title{Optimal Dividends Under Model Uncertainty}
\date{\today}

\author[P. Chakraborty]{Prakash Chakraborty}
\address{Department of Mathematics\\
University of Michigan\\
Ann Arbor, MI 48109\\
United States
}
\email{cprakash@umich.edu}
\author[A. Cohen]{Asaf Cohen }
\address{Department of Mathematics\\
University of Michigan\\
Ann Arbor, MI 48109\\
United States
}
\email{shloshim@gmail.com }

\author[V. R. Young]{Virginia R. Young}
\address{Department of Mathematics\\
University of Michigan\\
Ann Arbor, MI 48109\\
United States
}
\email{vryoung@umich.edu }

\thanks{A. Cohen acknowledges the financial support of the National Science Foundation (DMS-2006305).
 V. R. Young thanks the Cecil J. and Ethel M. Nesbitt Chair of Actuarial Mathematics for financial support.}
\date{\today}

\begin{document}

\maketitle

\begin{abstract}
We consider a diffusive model for optimally distributing dividends, while allowing for Knightian model ambiguity concerning the drift of the surplus process.  We show that the value function is the unique solution of a non-linear Hamilton--Jacobi--Bellman variational inequality.  In addition, this value function embodies a unique optimal threshold strategy for the insurer's surplus, thereby making it the smooth pasting of a non-linear and linear part at the location of the threshold.  Furthermore, we obtain continuity and monotonicity of the value function and the threshold strategy with respect to the parameter that measures ambiguity of our model.

\medskip


\medskip


\medskip


\end{abstract}

\section{Introduction}\label{sec_1}

Optimal dividend payment has been a classical problem in insurance mathematics since the seminal work of De Finetti \cite{DeF1957}.  Gerber \cite{G1972} showed that, for the classical Cram\'er--Lundberg risk model, the optimal dividend strategy is a band strategy.  Since those early paper, many insurance economists have studied the optimal dividend problem.  For example, Asmussen and Taksar \cite{asm-tak} considered a risk-neutral optimal dividend problem in the diffusive setup.  They exploited the linear structure of the Hamilton--Jacobi--Bellman variational inequality (HJB-VI) and provided explicit solutions for the value function and the optimal threshold. Azcue and Muler \cite{azc-mul2010} studied an optimal control problem of dividend payments and investment of the surplus in Black-Scholes market. In their case, the uncontrolled reserve follows the Cram\'er--Lundberg risk model, which leads to a jump-diffusion problem. They used methods of viscosity solutions to characterize the value function and to show that the optimal dividend control has a band structure.  Cohen and Young \cite{coh-you2021} determined the degree to which the diffusion approximation serves as a valid approximation for the Cram\'er--Lundberg model.

In the classical models considered above, it is assumed that the insurer has complete information about the dynamics. However, in reality it is rare that a decision maker has complete information about the parameters of the model. We model this uncertainty by including an adverse player who chooses a worst-case scenario.  This ambiguity is one example of what is called in the literature as {\it Knightian uncertainty}.   For further research that involves such uncertainty, we refer the reader to 
Maenhout \cite{maenhout2004robust}, Hansen \cite{Hansen2006}, Hansen and Sargent \cite{han-sar}, Bayraktar and Zhang \cite{bay-zha}, Neufeld and Nutz \cite{neu-nut2018}, Lam \cite{MR3544795}, Cohen \cite{Cohen2019a, Cohen2019b}, Cohen and Saha \cite{Cohen2020}, and Cohen, Hening, and Sun \cite{coh-hen-sun2021}.

In this paper, we formulate and analyze a problem of optimal dividends with model uncertainty. We assume that there is a {\it reference probability measure} $\PP$, under which the dynamics of the surplus  process, before paying dividends is
$$
X_t = x + mt + \sigma W_t, \qquad t\ge 0,
$$
in which $(W_t)_{t\ge 0}$ is a $\PP$-standard Brownian motion.  To account for uncertainty in the value of $m$, the insurer considers a class $\calQ(x)$ of probability measures that are equivalent to $\PP$, satisfying further conditions, and chooses a dividend payment process $D$ that maximizes the following payoff criteria: 
\begin{align*}
\inf_{\Q \in \calQ(x)}\E^\Q\Big[\int_0^\tau e^{-\varrho t}f(X_t)dt+\int_0^\tau e^{-\varrho t}dD_t +\frac{1}{\kappa}L^\varrho(\Q\|\PP)\Big].
\end{align*}
We look at this robust problem as a game between the insurer (maximizer) and an adverse player (minimizer).  The infimum is taken over the class of measures $\calQ(x)$.  The parameter $\varrho>0$ is the discount factor and $\tau$ is the (random) time of ruin. The first integral is a running reward, and the second one represents the dividend payments. Finally, the last term penalizes the adverse player for deviating from the reference measure: $L^\varrho(\Q\|\PP)$ is the Kullback--Leibler divergence that measures how much $\Q$ deviates from $\PP$, and $\kappa>0$ measure the level of ambiguity, with increasing values of $\kap$ corresponding to increasing ambiguity.

We characterize the value function and the Stackelberg equilibrium of this game, which consists of the optimal dividend strategy and the optimal response of the adverse player. Specifically, we show that the value function is the unique smooth solution of the relevant HJB-VI equation, and that the optimal dividend payment strategy is a threshold strategy, in which the threshold is determined by the HJB-VI as well.  In particular, we show that, for large values of $\kappa$, the threshold is 0, that is, dividends are paid immediately.  We summarize these properties in the main theorem of the paper (Theorem \ref{thm_41}).  For this, we set up the HJB-VI equation and transform it to a free-boundary problem by hypothesizing that the optimal strategy is a threshold strategy.  The penalty term from the payoff function translates to a quadratic term in the HJB-VI equation, which breaks the linearity of the differential equation.  As a consequence, an explicit solution is out of reach, and it is not clear if there is a smooth solution to the HJB-VI. To show that, indeed, a smooth solution exists, we use the {\it shooting method}.   On a high level, the shooting method  solves boundary-value problems by using a class of parameterized initial-value problems. 

We also analyze the dependence of the game on the ambiguity parameter $\kappa$. Namely, we show continuity and monotonicity properties of the value function and the optimal dividend-threshold with respect to $\kappa$. Finally, we show that when the ambiguity parameter $\kappa\to0^+$, the problem converges to the classical risk-neutral optimal dividend problem studied by Asmussen and Taksar \cite{asm-tak}.

In summary, our main contributions are as follows. We
\begin{itemize} \itemsep0em
	\item formulate a (diffusive) dividend problem with model uncertainty; 
	\item show that the value function solves a non-linear HJB-VI (Theorem \ref{thm_41});
	\item show that there is an optimal threshold strategy (Theorem \ref{thm_41});
	\item analyze the dependence of the value function and the optimal strategy on the ambiguity parameter (Theorems \ref{thm_45} and \ref{thm_46}).
\end{itemize}

The paper is organized as follows. In Section \ref{sec_2}, we motivate and present the stochastic differential game. Next, in Section \ref{sec_3}, we provide the HJB-VI for the value function and prove that the value function is the unique smooth solution of the HJB-VI with boundary condition.  Moreover, we show that the maximizer has an optimal unique threshold strategy.  Finally, in Section \ref{sec_4}, we study the dependency of the solution on the ambiguity parameter. 

\section{The stochastic model}\label{sec_2}

In this section, we present the ingredients for the optimal dividend problem under model uncertainty.  

\subsection{Stochastic game (SG)}\label{sec:LSG}
 
Let $(\Omega, \calF, \F = \{\calF_t \}_{t \ge 0}, \PP)$ be a filtered probability space that supports a one-dimensional standard Brownian motion $W$. We consider the following time-homogeneous dynamics for an insurer's uncontrolled surplus process $\hat X$:
 $$
 d \hat{X}_t = m dt + \sigma dW_t, \quad t\geq 0,
 $$
 with $\hat{X}_{0^-}=x \geq 0$.  The insurer chooses its dividend strategy to maximize a robust payoff functional that accounts for the uncertainty about the underlying model.  
 
\begin{definition}[admissible strategies]\label{def:adm-strat}
An {\rm admissible strategy for the maximizer} for any initial state $x \in \R_+$ is an $\F$-adapted, non-decreasing process $D$ taking values in $\R_+$ with right-continuous with left limits $($RCLL$)$ sample paths, with 
\begin{align}\label{eq:X}
dX_t = mdt + \sig dW_t - dD_t, \quad t\ge 0,
\end{align}
and $X_{0^-} = x \ge 0$, and with $D_t-D_{t^-}\le X_{t^-}$.  Let $\calA(x)$ denote the collection of admissible strategies $D$ with initial condition $x \geq 0$.
	
An {\rm admissible strategy for the minimizer} $($which we also call the {\rm adverse player}$)$ is a probability measure $\Q$, equivalent to $\PP$ on $(\Omega, \calF, \F)$, which is defined by 
\begin{align}\label{eq:radon}
\frac{d \Q}{d\PP}(t) = \exp\left\{\int_0^t \xi_s d W_s - \frac{1}{2}\int_0^t \xi^2_s ds\right\},\quad t\in\R_+,
\end{align}
for some $\F$-progressively measurable process $\xi$ satisfying 
\begin{align}\label{eq:novikov}
&\E^{\PP}\left[\int_0^\iy e^{-\varrho s} \xi^2_s \, ds \right]<\iy\quad\text{and}\quad\E^{\PP}\left[e^{\frac{1}{2}\int_0^t \xi^2_s \, ds}\right]<\iy, \quad t\in\R_+.
\end{align}
We call $\xi$ the {\rm Girsanov kernel} of $\Q$.  Let $\calQ(x)$ denote the collection of admissible strategies $Q$ with initial condition $x \geq 0$.  \qed
\end{definition}

\begin{remark}
 The conditions in \eqref{eq:novikov} ensure that \eqref{eq:radon} is a uniformly integrable martingale and that the discounted Kullback--Leibler divergence $($or relative entropy$)$ between $\Q$ and $\PP$ is well-defined in the sequel.  \qed
 \end{remark}

\noindent{\bf The payoff function.} 
Define the time of ruin $\tau$ by
\begin{align}\notag
\tau := \inf\{t\ge 0:X_t = 0\},
\end{align}
for $X_{0^-} = x \ge 0$.
The payoff associated with the initial condition $x$ and the strategy profile $(D, \Q)$ is given by
\begin{align}\label{307}
	 J(x, D, \Q; \kap):= \;&
	\E^{ \Q}\left[\int_0^{\tau} e^{-\varrho t}\big(f(X_t)dt + d D_t\big) \right] + \dfrac{1}{\kap}\, L^\varrho(\Q\|\PP),
\end{align}
in which $L^\varrho$ is the so-called (discounted) {\it Kullback--Leibler divergence}:
\begin{align}\label{eq:Lvarrho_d}
L^\varrho( \Q\|\PP)&:=
\E^{ \Q}\left[\int_0^{\infty}\varrho e^{-\varrho t}\ln\left(\frac{d \Q}{d\PP}(t)\right)dt\right],
\end{align}
and $\kap \ge 0$ measures the insurer's degree of ambiguity concerning $\PP$.  Increasing $\kap$ corresponds to increasing uncertainty.  In \eqref{307}, $f$ is a non-decreasing running-reward function.

We rewrite $L^\varrho( \Q\|\PP)$ in \eqref{eq:Lvarrho_d} as follows:
\begin{align*}
L^\varrho( \Q\|\PP) &= \E^{ \Q}\left[\int_0^{\iy}\varrho e^{-\varrho t}\ln\left(\frac{d \Q}{d\PP}(t)\right)dt\right] \\
&= \E^{ \Q}\left[\int_0^{\iy}\varrho e^{-\varrho t} \left\{\int_0^t \xi_s d W_s - \frac{1}{2}\int_0^t \xi^2_s \, ds\right\} dt\right] \\
&= \E^{ \Q}\left[\int_0^{\iy}\varrho e^{-\varrho t} \left\{\int_0^t \xi_s (d W_s - \xi_s ds) + \frac{1}{2}\int_0^t \xi^2_s \, ds\right\} dt\right] \\
&= \E^{ \Q}\left[\dfrac{1}{2} \int_0^{\iy}\varrho e^{-\varrho t} \int_0^t \xi^2_s \, ds \, dt\right] \\
&= \E^{ \Q}\left[\dfrac{1}{2} \int_0^{\iy} \left(\int_s^{\iy} \varrho e^{-\varrho t} dt \right) \xi^2_s \, ds \right] \\
&= \E^{ \Q}\left[\dfrac{1}{2} \int_0^{\iy}e^{-\varrho t} \xi^2_t \, dt \right],
\end{align*}
in which the fourth line follows because 
\begin{equation}\label{eq:W^Q}
W_t^{\Q} := W_t - \int_0^t \xi_s ds,\qquad t \ge 0,
\end{equation}
 is a $\Q$-Brownian motion.  
The value function is defined by
\begin{align}\label{eq:value}
	 V(x; \kap) = \sup_{D\in \calA(x)}\;\inf_{ \Q\in \calQ( x)}\; J( x, D, \Q; \kap).
\end{align}
Note that, under $\Q$, $X$ follows the process
\begin{equation}\label{eq:X_Q}
dX_t = (m + \sig \xi_t) dt + \sig dW^{\Q}_t - dD_t, \quad t\ge 0.
\end{equation}

\begin{remark}\label{rem:equivalent_form}
Given an admissible strategy of the adverse player $\Q$ with Girsanov's kernel $\xi$, define the admissible strategy $\Q^\tau$ with the Girsanov's kernel $\xi^\tau$, satisfying $\xi^\tau_t=\xi_t$ for $t\in[0,\tau]$, and $\xi^\tau_t$ for $t>\tau$. Then, the distribution of $X$ and $D$ until time $\tau$ are the same under both measures $\Q$ and $\Q^\tau$. On the other hand, 
$$
L^\varrho( \Q\|\PP)\le L^\varrho( \Q^\tau\|\PP) = \frac{1}{2} \, \E^{ \Q} \bigg[ \int_0^{\tau}e^{-\varrho t} (\xi^\tau_t)^2 \, dt \bigg].
$$
Therefore, the adverse player would prefer to use $\Q^\tau$ over $\Q$. As a result, we define
\begin{align}\label{cost2}
J(x, D, \Q; \kap)= \;&
	\E^{ \Q}\left[\int_0^{\tau} e^{-\varrho t}\Big\{\Big(f(X_t)+\frac{1}{2\kappa}\xi^2_t\Big)dt + d D_t\Big\} \right] .
\end{align} 
\end{remark}

We use the notation $J(x,D;0)$ and $V(x;0)$ to denote, respectively, the payoff and the value function for the risk-neutral problem. That is,
\begin{align*}
	 J(x, D; 0)= \;&
	\E^{ \PP}\left[\int_0^{\tau} e^{-\varrho t}\Big\{f(X_t) dt + d D_t \Big\} \right] = J(x, D, \PP; \kappa),
\end{align*}
and
\begin{equation}\label{eq:V-risk-neutral}
V(x;0) = \sup_{D \in {\calA}(x)}J(x,D;0).
\end{equation}

\begin{remark}
For $\kappa\sim0^+$, the penalty for deviating from the reference measure is very large. As we will see in the main theorem, the adverse player's optimal strategy is (stochastic and) of order $\kappa$. Therefore, we have convergence to the risk-neutral problem as $\kappa \to 0^+$. In addition, when $\kappa \to \infty$, the minimizer can choose the process $\xi_t$ go to infinity at a rate more slowly than that of $\kappa$, such that the time of ruin $\tau$ for the process $X_t$ in \eqref{eq:X} converges to $0$ under $\mathbb{Q}$. Thus, as $\kappa \to \infty$, $V$ heuristically should converge to $0$. We prove both of these results rigorously in the sequel.  \qed
\end{remark}

\section{Solution of the stochastic game}\label{sec_3}

\subsection{Threshold strategies}\label{sec_41a}

We show that the optimal strategy of the maximizer is a {\it threshold strategy}.  To rigorously define such a strategy, we use the {\it Skorokhod map on an interval}.  Fix $\beta\in\R_+$; then, for any RCLL function $\eta:\R_+\to\R$, there exist RCLL functions $\chi, \zeta :\R_+\to\R$ that satisfy the following properties:
\begin{enumerate}
\item[(i)] For every $t\geq 0$, $\chi_t = \eta_t-\zeta_t$.

\item[(ii)] $\zeta$ is non-decreasing, with $\zeta_{0^-}=0$ and
\begin{align}\notag
	\int_0^{\infty}
	1_{(-\infty,\beta]}(\chi_t)d\zeta_t=0,
\end{align}
Given $\beta \ge 0$ and RCLL $\eta :\R_+\to \R$, the pair $(\chi,\zeta)$ is unique on $\R_+$. 
Let $\Gamma_{\beta}(\eta)=(\Gamma_{\beta}^1,\Gamma_{\beta}^2)(\eta)$ denote the ordered pair $(\chi,\zeta)$. 
\end{enumerate}

\medskip

\noindent  The following continuity property is well known; see, for example, Kruk et al. \cite{Kruk2007}.

\begin{lemma}\label{lem_Skorokhod}
There exists a constant $c_S>0$ such that for every $t \ge 0$, $\beta\in\R_+$, and RCLL functions $\eta,\tilde\eta:\R_+\to\R$,
\begin{equation}\notag
	\sup_{s\in[0,t]} \left\{ \left | \Gam^1_\beta(\eta_s) - \Gam^1_\beta(\tilde \eta_s) \right| + \left| \Gam^2_\beta(\eta_s) - \Gam^2_\beta(\tilde \eta_s) \right| \right\}
	\le c_S\sup_{s\in[0,t]}|\eta_s-\tilde\eta_{s}|.
\end{equation}
\end{lemma}

\begin{definition}\label{def_Skorokhod}
	Fix $x, \beta \geq 0$. 
	The strategy $D^\beta$ is called a $\beta$-threshold strategy if $(X_t ,D_t^\beta)=\Gamma_{\beta}(x+m\cdot+\sigma W_{\cdot})_t$, for all $t\ge 0$, with $X_{0^-} = x$.   \qed
\end{definition}

\noindent One can easily verify that any $\beta$-threshold strategy $D^\beta$ is admissible.  Essentially, $D^\beta$ pays all surplus in excess of $\beta$ as dividends.

\subsection{The HJB variational inequality and the value function}\label{sec_41}

For $\kap > 0$, we anticipate that $V$ in \eqref{eq:value} solves the following HJB-VI with boundary condition:
\begin{equation}\label{405}
	\begin{cases} 
		\left[ \inf \limits_{\xi\in\R}\left\{\dfrac{1}{2}\sigma^2 \phi ''(x)+(m+\sigma \xi)\phi'(x)-\varrho \phi (x)+f(x) + \dfrac{1}{2\kappa} \xi^2\right\} \right]\vee \big[1-\phi'(x) \big]=0, &x> 0, \\ 
		\phi(0)=0.
	\end{cases} 
\end{equation}
The coefficient of $\phi'(x)$ arises from the $\Q$-drift of $X$ in \eqref{eq:X_Q}; recall that $V$ is a $\Q$-expectation.  By substituting the optimal solution from the $\inf_{\xi\in\R}$, namely, {$\xi^*=-\kappa\sigma \phi'(x)$}, \eqref{405} becomes
\begin{align}\tag{HJB($\kappa$)}\label{HJB}
	\begin{cases} 
		\big[\phi''(x)+H(x,\phi(x),\phi'(x)) \big]\vee  \big[1-\phi'(x) \big]=0, &x\in(0,\infty), \\ 
		\phi(0)=0,
	\end{cases} 
\end{align}
in which
\begin{align}\label{eq:H}
H(x,y,z) := \frac{2}{\sigma^2} \left(mz-\frac{1}{2}\sigma^2\kappa z^2-\vr y+f(x) \right).
\end{align}
Moreover, \eqref{HJB} makes sense when $\kap = 0$.  As we prove below, \eqref{HJB} admits a unique solution in $\calC^2(\R_+)$ for all $k \ge 0$, and it solves
\begin{align}\label{eq:phi_beta}
	\begin{cases}
		\phi_\beta''(x)+H(x,\phi_\beta(x),\phi_\beta'(x))=0,\qquad &0\le x\le\beta,\\
		1-\phi_{\beta}'(x)=0,\qquad &\beta\le x,\\
		\phi_{\beta}(0)=0,
	\end{cases}
\end{align}
for some $\beta\in\R_+$, which (together with a verification result) implies that the optimal dividend strategy is a threshold strategy.

\begin{remark}
Later, we will identify $\beta$ via the smooth-pasting $($free-boundary$)$ condition $\phi_{\beta}''(\beta)=0$, which is consistent with $\phi$ being twice continuously differentiable at $x = \beta$.  \qed
\end{remark}

\begin{remark}
Note that, when $\kappa=0$ and $f\equiv 0$, \eqref{HJB} coincides with the HJB-VI given in Asmussen and Taksar \cite[Equations (3.10)--(3.11)]{asm-tak}.  \qed
\end{remark}

From the minimization in \eqref{405}, we define a candidate strategy for the adverse player:
For any $\phi \in\calC^2(\R_+)$ and $t\in\R_+$, set 
\begin{align}\label{eq:Q}
	\xi^{\phi}_t:&=\argmin_{\xi\in\R}\left\{\frac{1}{2}\sigma^2\phi''(X_t)+(m+\sigma \xi)\phi'(X_t)-\vr \phi(X_t)+f(X_t)+\frac{1}{2\kappa} \xi^2\right\} \notag \\ 
	&=
	-\kappa \sigma \phi'(X_t).  
\end{align}
$\xi^{\phi} = \big\{ \xi^{\phi}_t \big\}_{t \ge 0}$ is an $\F$-progressively measurable process because, from Definition \ref{def:adm-strat}, $D$ and $W$ are.  If $\xi^\phi$ satisfies the conditions \eqref{eq:novikov}, let $\Q^{\phi}$ denote the measure associated with the Girsanov kernel $\xi^{\phi}$. 

In what follows, we also need some mild constraints on the running-reward function $f$. Specifically, we impose the following assumption for the rest of this section:

\begin{assumption}\label{as:f-lips}
The function $f$ is non-negative, non-decreasing, and Lipschitz with corresponding coefficient $\varrho - \delta$ for some $\delta>0$. That is,
$$
| f(x) - f(y) | \leq (\varrho-\delta) |x-y|.
$$
In addition, we set $f(0)=0$ to account for zero reward when there is no surplus.  \qed
\end{assumption}

Under this assumption, we have the following theorem.

\begin{theorem}\label{thm_41}
	For any $\kappa\in[0,\iy)$, the following hold: 
\begin{enumerate}
\item[$(1)$] The value function given by \eqref{eq:value} is the unique $\calC^2(\R_+)$ solution of \eqref{HJB}. Moreover, $V$ is concave and solves \eqref{eq:phi_beta} for some $\beta \ge 0$.

\item[$(2)$] Let $\beta_{\kappa}$ denote the value of $\beta \ge 0$ associated with $\kappa$.  Then, $\beta_{\kappa}$ is unique.  Furthermore, for $2m \leq \sigma^2\kappa$, we have $\beta_\kappa=0$, and for $2m > \sigma^2\kappa$, we have $0 < \beta_\kappa \leq \frac{m}{\delta}$. 

\item[$(3)$] 
Let $\xi^{V(\cdot;\kappa)}$ be given by \eqref{eq:Q} with $\phi=V(\cdot;\kappa)$. Then, $ \xi^{V(\cdot;\kappa)}$ is bounded; hence, the conditions in \eqref{eq:novikov} hold. Let $\Q^\kappa: = \Q^{V(\cdot;\kappa)}$ be the measure associated with the Girsanov kernel $\xi^{V(\cdot;\kappa)}$, and let $D^{\beta_\kappa}$ denote the $\beta_\kappa$-threshold strategy. Then, the couple $(D^{\beta_\kappa}, \Q^\kappa)$ forms a Stackelberg equilibrium in the following sense:
\begin{equation}\label{eq:equil}
V(x; \kap) = J(x,D^{\kappa},\Q^\kappa;\kappa) = \inf_{\Q \in \mathcal Q(x)}J(x,D^{\kappa},\Q;\kappa) = \sup_{D \in \mathcal{A}(x)}J(x,D,\Q^{\kappa};\kappa).
\end{equation}
\end{enumerate}
\end{theorem}
The proof of Theorem \ref{thm_41} follows from the next four propositions; see Section~\ref{sec:proofs_1-4} for their proofs. 

\begin{proposition}\label{prop_41}
Suppose $\phi \in \calC^2(\R_+)$ solves \eqref{HJB}, with $\phi'$ bounded. Then, for $\kappa>0$,
	\begin{align}\label{eq:prop_41_kap}
		\phi(x)\geq \sup_{D \in \calA(x)} J(x,D,\Q^{\phi};\kappa),\qquad x\in[0,\infty),
	\end{align}
and for $\kappa=0$,
	\begin{align}\label{eq:prop_41_0}
		\phi(x)\geq \sup_{D \in \calA(x)} J(x,D;\kappa),\qquad x\in[0,\infty).
	\end{align}
As a consequence, $\phi \geq V$.   \qed
\end{proposition}  

\begin{proposition}\label{prop_42}
Suppose $\phi_{\beta} \in\calC^2(\R_+)$ solves \eqref{eq:phi_beta} for some $\beta\ge 0$. Let $D^{\beta}$ be the corresponding $\beta$-threshold strategy.
Then, for $\kappa > 0$,
	\begin{align}\label{415}
		\phi_{\beta}(x)&\le\inf_{\Q\in\calQ(x)}J(x,D^{\beta},\Q;\kappa),  \qquad x \in [0, \infty),
	\end{align}
and for $\kappa = 0$,
\begin{align}\label{415a}
\phi_{\beta}(x) \leq J(x, D^{\beta};0), \qquad x \in [0, \infty).
\end{align}
As a consequence, $\phi_{\beta} \leq V$.  \qed
\end{proposition}

\begin{proposition}\label{prop_43}
	For every $\kappa \geq 0$, \eqref{HJB} admits a unique $\calC^2(\R_+)$ solution with a bounded derivative, uniformly in $\kappa$.  Moreover, the solution also solves \eqref{eq:phi_beta}.  Let $\beta_{\kappa}$ denote the value of $\beta \ge 0$ associated with $\kappa$.  Then, for $2m \leq \sigma^2\kappa$, we have $\beta_\kappa=0$, and for $2m > \sigma^2\kappa$, we have $0 < \beta_\kappa \leq \frac{m}{\delta}$.  \qed 
\end{proposition}

\begin{proposition}\label{prop_44}
	For every $\kappa \ge 0$, there is a unique $\beta$-threshold optimal strategy.  \qed
\end{proposition}

\begin{proposition}\label{prop_45}
The value function given by \eqref{eq:value} is concave, that is, for any $\kappa \geq 0$,
$$
V''(x ; \kappa) \leq 0,  \quad x \in [0, \infty).
$$
\end{proposition}

\begin{proof}[Proof of Theorem \ref{thm_41}.] 
From Proposition{s} \ref{prop_43} {and \ref{prop_45}}, for every $\kappa \ge 0$, \eqref{HJB} admits a unique {concave,} $\calC^2(\R_+)$ solution that also solves \eqref{eq:phi_beta} for some $\beta_{\kappa} \geq 0$, which is unique by Proposition~\ref{prop_44}. Let $g_{\beta_\kappa}$ denote this solution. From Proposition \ref{prop_42}, we have
$$
g_{\beta_{\kappa}}(\cdot) \leq \inf_{\Q \in \calQ(x)} J(\cdot, D^{\beta_{\kappa}}, \Q;\kappa) \leq V(\cdot;\kappa).
$$
On the other hand, from Proposition~\ref{prop_41},
$$
g_{\beta_{\kappa}}(\cdot) \geq \sup_{D \in \calA(x)} J(\cdot, D, \Q^{\kappa}; \kappa) \geq V(\cdot ; \kappa).
$$
By combining the above two inequalities, we deduce $V(\cdot; \kappa) = g_{\beta_{\kappa}}(\cdot)$, as well as the equilibrium relation in \eqref{eq:equil}. From Proposition~\ref{prop_43}, we also have $\beta_{\kappa} = 0$ if $ 2m \leq \sigma^2 \kappa$.
\end{proof}

\begin{remark}
Note that when the insurer has large enough ambiguity aversion, namely, $\kap \ge 2m/\sig^2$, then it is optimal to pay out all the surplus as dividends, and $V(x; \kap) = x$ for all $x \ge 0$ in this case.  \qed 
\end{remark}

\subsection{Proofs of Propositions \ref{prop_41}--{\ref{prop_45}}}\label{sec:proofs_1-4}

\begin{proof}[Proof of Proposition \ref{prop_41}.]  
Suppose $\phi \in \calC^2(\R_+)$ solves \eqref{HJB} with bounded first derivative.  For arbitrary $D \in \calA(x)$ and $\Q \in \calQ(x)$, It\^o's lemma applied to $e^{-\vr (\tau\wedge t)}{\phi}(X_{\tau\wedge t})$ and $\Q$-expectation give us
\begin{align}\label{408}
	\E^{\Q}\left[e^{-\vr (\tau\wedge t)}{\phi}(X_{\tau\wedge t})\right] &= {\phi}(x)+\E^{\Q}\left[\int_0^{\tau\wedge t}e^{-\varrho s}\left(\frac{1}{2}\sigma^2 {\phi}''(X_s)+(m+\sigma\xi_s){\phi}'(X_s)-\varrho {\phi}(X_s)\right)ds\right] \notag \\
	&\quad -\E^{\Q}\left[\int_0^{\tau\wedge t} e^{-\varrho s} {\phi}'(X_s) dD^{\texttt{C}}_u \right] +\E^{\Q}\left[\sum_{0\le s\le (\tau\wedge t)}e^{-\varrho s}{\Delta {\phi}(X)}_{s}\right],
\end{align}
in which $\xi$ is the Girsanov kernel of $\Q$, and the process $D^{\texttt{C}}$ denotes the continuous part of $D$. Consider \eqref{408} with $\xi=\xi^{\phi}$ defined in \eqref{eq:Q} and with the associated measure $\Q=\Q^{\phi}$. 
Because ${\phi}$ solves \eqref{HJB}, we obtain
\begin{align}\label{410}
	\E^{\Q^{\phi}}\left[e^{-\varrho (\tau\wedge t)}\phi(X_{\tau\wedge t})\right]&\leq {\phi}(x)-\E^{\Q^{\phi}}\left[\int_0^{{\tau\wedge t}} e^{-\varrho s}\Big\{\Big(f(X_s)+\frac{1}{2\kappa}{(\xi^{\phi}_s)^2 \one_{\{\kappa > 0\}} }\Big)ds+dD_s\Big\}\right]\\\notag
	&\quad +\E^{\Q^{\phi}}\left[\sum_{0\le s\le (\tau\wedge t)}e^{-\varrho s}\big({\Delta {\phi}(X)}_{s}+\Delta D_s \big)\right].
\end{align}
Because $\Delta X_s = -\Delta D_s$, we have
\begin{align}\label{411}
	{\Delta {\phi}(X)}_{s}+\Delta D_s &= {\phi}(X_s-\Delta D_s) - {\phi}(X_s) + \Delta D_s\\\notag
	&=\int_{X_s - \Delta D_s}^{X_s}(1-{\phi}'(u))du  \leq 0.
\end{align}
By combining \eqref{410}--\eqref{411}, we get 
\begin{align}\label{eq:prop-3.1-prefinal1}
	\E^{\Q^{\phi}}\left[e^{-\varrho (\tau\wedge t)}{\phi}(X_t)\right]+\E^{\Q^{\phi}}\left[\int_0^{(\tau\wedge t)} e^{-\varrho s}\left\{ \left({f}(X_s)+\frac{1}{2\kappa}(\xi^{\phi}_s)^2 \one_{\{\kappa > 0\}} \right)ds+dD_s\right\}ds\right]\leq {\phi}(x).
\end{align}

Note that $\phi$ is non-negative; indeed, the solution of \eqref{HJB} satisfies $\phi'(x) \ge 1$ with $\phi(0) = 0$.  Because $\phi$ is non-negative, we can omit the leftmost term in \eqref{eq:prop-3.1-prefinal1} while maintaining the inequality. The non-negativity of the terms within the integral together with the monotone convergence theorem (as $t\to\iy$) imply 
\begin{align}\label{eq:prop-3.1-prefinal}
J(x,D,\Q^\phi;\kappa)=\E^{\Q^{\phi}}\left[\int_0^{\tau} e^{-\varrho s}\left\{ \left({f}(X_s)+\frac{1}{2\kappa}(\xi^{\phi}_s)^2 \one_{\{ \kappa > 0\}} \right)ds+dD_s\right\}\right]\leq {\phi}(x),
\end{align}
from which we deduce inequality \eqref{eq:prop_41_kap}.  Finally, because $\xi^{\phi} = 0$ when $\kappa = 0$, we have $\Q^{\phi} = \PP$ when $\kappa = 0$, and inequality \eqref{eq:prop_41_0} follows.
\end{proof}

\begin{proof}[Proof of Proposition \ref{prop_42}]
We split the proof into two cases: $x\in[0,\beta]$ and $x\in(\beta, \infty)$.
\vspace{5pt} \\
{\it Case 1:} 
Fix $x\in[0,\beta]$, and choose an arbitrary $\Q\in\calQ(x)$ with Girsanov kernel $\xi$. Without loss of generality, we may assume that 
\begin{align}\notag
	J(x,D^{\beta},\Q;\kappa)\le \phi_{\beta}(x)+1.
	\end{align}
Otherwise, the desired inequality in \eqref{415} follows immediately. As a consequence,
\begin{align}\label{bound_phi}
\E^{\Q}\left[\int_0^{\tau} e^{-\varrho s}\frac{1}{2\kappa}(\xi_s)^2 \one_{\{ \kappa > 0\}} ds\right]\leq {\phi}(x)+1.
\end{align}
By basic properties of quadratic functions, for every $t \geq 0$: 
\begin{align}\label{415c}
	&\frac{1}{2}\sigma^2\phi_\beta''(X_t) + \big(m+\sigma {\xi}_t  \one_{\{ \kappa > 0\}} \big) \phi_\beta'(X_t)-\varrho \phi_\beta(X_t)+f(X_t)+\frac{1}{2\kappa}({\xi}_t)^2 \one_{\{ \kappa > 0\}} \\\notag
	&\quad
\geq
\frac{1}{2}\sigma^2\left(\phi_\beta''(X_t)+H(X_t,\phi_\beta(X_t),\phi_\beta'(X_t))\right)=0,
\end{align}
in which $H$ is given in \eqref{eq:H}, and the equality follows from the first equation in \eqref{eq:phi_beta}.

Note that, because $x\in[0,\beta]$ and $D^{\beta}$ is a $\beta$-threshold strategy, $D^{\beta}$ has no jumps (that is, $(D^\beta)^{\texttt{C}}=D^\beta$). From \eqref{408} and \eqref{415c}, we have
\begin{align}\label{415d}
	\E^{\Q}\left[e^{-\varrho (\tau\wedge t)}\phi_\beta(X_{\tau\wedge t})\right] &\geq \phi_\beta(x)-\E^{\Q}\left[\int_0^{\tau\wedge t}e^{-\varrho s}\left(f(X_s)+\frac{1}{2\kappa}{\xi}_s^2 \one_{\{ \kappa > 0\}} \right)ds\right]\\\notag
	&\quad -\E^{\Q}\left[\int_0^{\tau\wedge t}e^{-\varrho s}\phi_\beta'(X_s)dD^\beta_s\right].
\end{align}
By using $\phi_\beta'(\beta)=1$ from \eqref{eq:phi_beta} and by noting $dD^\beta_s \ne 0$ only when $X_s = \beta$, we obtain
\begin{align}\label{415e}
	\E^{\Q}\left[e^{-\varrho (\tau\wedge t)}\phi_\beta(X_{\tau\wedge t})\right] \geq \phi_\beta(x) &-\E^{\Q}\left[\int_0^{\tau\wedge t}e^{-\varrho s}\left\{\left(f(X_s)+\frac{1}{2\kappa}{\xi}_s^2 \one_{\{ \kappa > 0\}} \right)ds+dD^\beta_s\right\}\right].
\end{align}
Next, we show that the left side of inequality \eqref{415e} vanishes as $t\to\iy$. The linear growth of $\phi_\beta$ as a solution of \eqref{eq:phi_beta} and the dynamics of $X$ imply 
\begin{align}\label{eq:sequence_a}
	\sup_{t\in\R_+} \E^{\Q} \left\{e^{-\varrho (\tau\wedge t)}\phi_\beta(X_{\tau\wedge t})\right\}
&\le C+
	\sup_{t\in\R_+}\E^{\Q}\left\{e^{-\varrho (\tau\wedge t)}X_{\tau\wedge t}\right\} \notag \\
&\le C+
	\sup_{t\in\R_+}\E^{\Q}\left\{e^{-\varrho (\tau\wedge t)}\left(x+m(t\wedge\tau)+\sigma \int_0^{\tau\wedge t}|\xi_s|ds\right)\right\}
\notag \\
&\le C+
\sigma \E^{\Q} \left\{\int_0^{\tau}e^{-\varrho s}|\xi_s|ds\right\},
\end{align}
in which $C$ is a positive constant that might change from one line to the next and that depends only on $x$, $m$, $\varrho$, and $\phi$. Now, the right side of \eqref{eq:sequence_a} is finite due to \eqref{bound_phi}. Therefore, the dominated convergence theorem together with the identity $X_\tau\one_{\{\tau<\iy\}}=0$ imply
\begin{align}\notag
\lim_{t\to\iy}	\E^{\Q}\left[e^{-\varrho (\tau\wedge t)}\phi_\beta(X_{\tau\wedge t})\one_{\{\tau<\iy\}}\right] = \E^{\Q}\Big[e^{-\varrho \tau}\phi_\beta(X_{\tau})\one_{\{\tau<\iy\}}\Big] = 0.
\end{align}
On the other hand, from \eqref{eq:sequence_a}, the Cauchy--Schwartz inequality, and \eqref{bound_phi}, we get
\begin{align}\notag
\lim_{t\to\iy}	\E^{\Q}\left[e^{-\varrho (\tau\wedge t)}\phi_\beta(X_{\tau\wedge t})\one_{\{\tau=\iy\}}\right] =0.
\end{align}
The last two limits together with the monotone convergence theorem and the non-negativity of the terms within the integral in the right side of \eqref{415e} imply 
\begin{align}\label{415f}
	\phi_\beta(x)\leq \inf_{\Q\in\calQ(x)}J(x,D^\beta,\Q;\kappa),
\end{align}
for all $\kap > 0$, and for $\kap = 0$,
\begin{equation}\label{415f_0}
\phi_\beta(x) \leq J(x, D^\beta;0).
\end{equation}

\vspace{5pt}
{\it Case 2:}
Consider now the case for which $x\in(\beta,\infty)$. From \eqref{eq:phi_beta},
$$
\phi_{\beta}(x)=(x-\beta)+\phi_{\beta}(\beta).
$$
Because the strategy $D^{\beta}$ starts with an instantaneous dividend payment of $x-\beta$, there is an immediate payoff of $(x-\beta)$ and, hence, for any $\Q\in\calQ(x)$, we have
$$
J(x,D^\beta,\Q;\kappa) = (x-\beta)+J(\beta,D^\beta,\Q;\kappa).
$$ 
From \eqref{415f} and \eqref{415f_0} applied to $x=\beta$, and the from last two equalities, we have the desired bound when $x\in[\beta,\iy)$.
\end{proof}

We next show that \eqref{HJB} admits a unique smooth solution for every $\kappa \in [0, \infty)$. However, the non-linearity in the differential equation prevents us from providing an explicit solution as, for example, provided in Asmussen and Taksar \cite[Theorem 3.2]{asm-tak} for a model with zero reward function and no  uncertainty in the model.  Instead, we will use the \emph{shooting method} to prove that there exists a unique smooth solution to \eqref{HJB}.  Cohen  \cite{Cohen2019a} applied this approach to a similar HJB equation, one that arises out of a queuing framework.  Specifically, for every $s \ge 1$, consider the following Cauchy problem:
\begin{align}\label{eq:g^(s)}
	\begin{cases}
		(\varphi^{(s)} )''(x) + H_F(x, \varphi^{(s)}(x), (\varphi^{(s)})'(x))=0,\qquad x > 0,\\
		\varphi^{(s)}(0)=0,\quad (\varphi^{(s)})'(0)=s,
	\end{cases}
\end{align}
in which
\begin{align}\label{HF}
	H_F(x,y,z):=H(x,y,F(z)).
\end{align}
Here, $F$ is a $\calC^1(\R)$ mollifier that satisfies
\begin{equation}\label{eq:F}
F(z)=z \text{ on } [-\bar{s},\bar{s}], \quad |F|\le 2\bar{s}, \quad \text{ and } \quad |F'|\le 1,
\end{equation}
in which $\bar{s} = \frac{2m}{\sigma^2 \kappa}$. Therefore, $F$ is Lipschitz continuous with constant $1$.  For example, we may choose the following mollifier:
\begin{align}\notag
	F(z)=\begin{cases}
		-(3/2) \bar{s}, &z< -2\bar{s},\\
		(1/2)\bar{s}+2z+z^2/(2\bar{s}), & -2\bar{s}\le z<-\bar{s},\\
		z, &-\bar{s}\le z\le \bar{s}, \\
		-(1/2)\bar{s}+2z-z^2/(2\bar{s}), &\bar{s}\le z<2\bar{s},\\
		(3/2) \bar{s}, &2\bar{s}\le z.
	\end{cases}
\end{align}
When $\kappa > 0$, the function $F$ and its derivative are bounded, and because the function $f$ is Lipschitz, $H_F$ is uniformly Lipschitz. Namely, there is a constant $L$ such that for every $(x,y,z),(x',y',z')\in\R_+\times\R\times\R$, one has
\begin{align}\label{ac05}
	\big| H_F(x,y,z)-H_F(x',y',z') \big| \le L\big( |x-x'|+|y-y'|+|z-z'| \big).
\end{align}
From Zaitsev and Polyanin \cite[Section 0.3.1]{Polyanin2003}, the Cauchy problem in \eqref{eq:g^(s)} admits a unique $\calC^2(\R_+)$ solution when $\kap > 0$.  For $\kappa = 0$, \eqref{eq:g^(s)} is a non-homogeneous linear equation, and existence of unique solution follows from classical results; see, for example, Zaitsev and Polyanin \cite{Polyanin2003}.  Define $\beta^{(s)}$ by
\begin{align}\label{ac09}
	\beta^{(s)}&:=\inf\big\{x > 0 : (\varphi^{(s)} )'(x)\leq 1 \big\}.
\end{align}
The smoothness of $\varphi^{(s)}$ implies that 
\begin{align}\label{ac04}
	\text{if $\;\beta^{(s)}<\infty, \;$ then $\; (\varphi^{(s)} )' (\beta^{(s)} ) = 1$.}
\end{align}

The following lemma provides some continuity properties that serve us in the proof of Proposition \ref{prop_43}. The continuity results are obtained for the norm 
$$
{\|\phi\|}_h := \int_0^{\infty} e^{- h u^2}  \left(1\wedge \sup_{x \in[0,u]} |\phi(x)| \right)du,
$$
which is defined for $\phi \in \calC(\R_+)$ and $h > 0$. 

\begin{lemma}\label{lem_42}
The function $s \mapsto (\varphi^{(s)}, (\varphi^{(s)} )', (\varphi^{(s)} )'' )$ on $[1, \iy)$ is continuous in the ${\|\cdot\|}_{h}$-norm topology for any $h > L/2$, in which $L$ is the Lipschitz constant in \eqref{ac05} for $H_F$. Moreover, the mapping $s\mapsto\beta^{(s)}$ is continuous on the set of $s \geq 1$ for which either $(i)$ $\beta^{(s)} = \infty$ or $(ii)$ the following two conditions hold: $\beta^{(s)} < \infty$ and $(\varphi^{(s)} )'' (\beta^{(s)} ) \ne 0$.  If $(ii)$ holds, we conclude that the mapping  $s \mapsto (\varphi^{(s)} (\beta^{(s)}), (\varphi^{(s)})'(\beta^{(s)}), (\varphi^{(s)} )'' (\beta^{(s)}))$ is also continuous. 
\end{lemma}
\begin{proof}
{\it Step 1: Continuity of $s \mapsto (\varphi^{(s)}, (\varphi^{(s)} )', (\varphi^{(s)} )'' )$.} 
Fix $s \ge 1$.  For every $\delta_1\in\R$  and $x\in\R_+$, set
\[
{\Delta \varphi}_{s,\delta_1}(x) := \left| \varphi^{(s+\delta_1)}(x) - \varphi^{(s)}(x)\right|,
\]
and define ${\Delta \varphi}_{s,\delta_1}'(x)$ and ${\Delta \varphi}_{s, \delta_1}''(x)$ similarly.   From the initial value of $\varphi^{(s)}$ in \eqref{eq:g^(s)}, we deduce
\begin{equation}\label{eq:cont-lem-1}
{\Delta \varphi}_{s, \delta_1} (x) = \left \vert \int_0^x \big( (\varphi^{(s+\delta_1)})'(y) - (\varphi^{(s)})'(y) \big) dy \right\vert \leq \int_0^x {\Delta \varphi}_{s, \delta_1}' (y) dy.
\end{equation}
Furthermore, we also have from \eqref{eq:g^(s)}:
\begin{align}\notag
	(\varphi^{(s)})'(x)&= s - \int_0^x H_F(y, \varphi^{(s)}(y), (\varphi^{(s)})'(y)) dy,\\\notag
	(\varphi^{(s+\delta_1)} )'(x)&= s + \delta_1 - \int_0^x H_F(y, \varphi^{(s+\delta_1)}(y), (\varphi^{(s+\delta_1)})'(y)) dy.
\end{align}
Then, from \eqref{ac05}, the following inequality holds, uniformly in $x$, $s$, and $\delta_1$:
\begin{align}\label{eq:cont-lem-2}
	{\Delta \varphi}_{s,\delta_1}'(x)&\le \left| \delta_1 \right| + L\int_0^x \left( {\Delta \varphi}_{s,\delta_1} (y) + {\Delta \varphi }_{s,\delta_1}'(y) \right) dy.
\end{align}
By substituting \eqref{eq:cont-lem-1} into \eqref{eq:cont-lem-2}, we obtain
\begin{align*}
{\Delta \varphi}_{s,\delta_1}'(x)&\le \left| \delta_1 \right| + L\int_0^x  \left(\int_0^y{\Delta \varphi}_{s,\delta_1}' (z) dz +{\Delta \varphi }_{s,\delta_1}'(y)\right)dy\\
&= \left| \delta_1 \right| + L\int_0^x  \left((x-y){\Delta \varphi}_{s,\delta_1}'(y)  +{\Delta \varphi }_{s,\delta_1}'(y)\right)dy \\
&\leq \left| \delta_1 \right| + L(1+x) \int_0^x  {\Delta \varphi}_{s,\delta_1}'(y)dy.
\end{align*}

Gr\"onwall's inequality implies (see, for example, Willett \cite[Theorem 0]{W1965})
\[
{\Delta \varphi}_{s, \delta_1}'(x) \leq  \left| \delta_1 \right| \left[ 1 + L x(1 + x) \exp \big( L x(1 + x/2) \big) \right].
\]
Thus,
\begin{equation}\label{eq:cont-lem-3a}
\sup_{x \in[0,u]} {\Delta \varphi}_{s, \delta_1}'(x) \le  \left| \delta_1 \right| \left[ 1 + L u(1 + u) \exp \big( L u(1 + u/2) \big) \right].
\end{equation}
Now we are ready to bound ${\| \Delta \varphi_{s, \delta_1}' \|}_h$ from above. Specifically,
\begin{align}\label{eq:cont-lem-s1a}
{\| {\Delta \varphi}_{s, \delta_1}' \|}_h &= \int_0^{\infty} e^{- h u^2} \left( 1 \wedge \sup_{x \in[0,u]} {\Delta \varphi}_{s, \delta_1}'(x) \right) du \notag \\
&\leq \left| \delta_1 \right| \int_0^{\infty} e^{-hu^2} \left[ 1 + L u(1 + u) e^{L u(1 + u/2)} \right] du = \left| \delta_1 \right| C_1,
\end{align}
for some $C_1 > 0$ that depends only on $L$ and $h$, as long as $h > L/2$.  Similarly, from \eqref{eq:cont-lem-1} 
$$
\sup_{x \in[0,u]} {\Delta \varphi}_{s , \delta_1} (x) \leq u \sup_{x \in[0,u]} {\Delta \varphi}_{s, \delta}'(x) \leq \left| \delta_1 \right| u \left[ 1 + L u(1 + u) \exp \big( L u(1 + u/2) \big) \right], 
$$
in which the last inequality follows from \eqref{eq:cont-lem-3a}. Therefore,
\begin{align}\label{eq:cont-lem-s2a}
{\| {\Delta \varphi}_{s, \delta_1} \|}_h &= \int_0^{\infty} e^{- h u^2} \left( 1 \wedge \sup_{x \in[0,u]} {\Delta \varphi}_{s, \delta_1}(x) \right) du \notag \\
&\leq  \left| \delta_1 \right| \int_0^{\infty} u e^{-hu^2} \left[ 1 + L u(1 + u) e^{L u(1 + u/2)} \right] du = \left| \delta_1 \right| C_2,
\end{align}
for some $C_2 > 0$ that depends only on $L$ and $h$, as long as $h > L/2$. A similar result holds for ${\|{\Delta \varphi}_{s, \delta_1}'' \|}_h$ by the relation $(\varphi^{(s)})''(x)=-H_F(x,\varphi^{(s)}(x), (\varphi^{(s)})'(x) )$ and the Lipschitz continuity of $H_F$ stated in \eqref{ac05}.  Indeed, for all $x \geq 0$,
 \begin{align*}
 \Delta \varphi_{s, \delta_1}''(x) &= \left| H_F(x, \varphi^{(s+\delta_1)}(x), (\varphi^{(s+\delta_1)})'(x) ) - H_F(x, \varphi^{(s)}(x), (\varphi^{(s)})'(x)) \right| \\ 
 &\leq L\left( \Delta \varphi_{s, \delta_1}(x) + \Delta \varphi_{s, \delta_1}'(x) \right),
 \end{align*}
 which implies 
 $$
1 \wedge  \sup_{x \in [0, u]}  \Delta \varphi_{s, \delta_1}''(x) \leq L \left( 1\wedge \sup_{x \in [0, u]}  \Delta \varphi_{s, \delta_1}(x)+1\wedge \sup_{x \in [0, u]} \Delta \varphi_{s, \delta_1}'(x) \right).
 $$
By integrating with respect to $e^{-hu^2}$ and by using \eqref{eq:cont-lem-s1a} and \eqref{eq:cont-lem-s2a}, we obtain that there exists  $C_3 > 0$ depending only on $h$ and $L$ such that
\begin{equation}\label{eq:cont-lem-s3a}
{\|{\Delta \varphi}_{s, \delta_1}'' \|}_h \leq \delta_1 C_3.
\end{equation}
The uniform bounds in $\delta_1$, given in \eqref{eq:cont-lem-s1a}, \eqref{eq:cont-lem-s2a}, and \eqref{eq:cont-lem-s3a} imply the continuity of the map $s \mapsto (\varphi^{(s)}, (\varphi^{(s)})', (\varphi^{(s)})'' )$  in the ${\|\cdot\|}_h$-norm topology.

\vspace{5pt}{\it Step 2: Continuity of $s\mapsto \beta^{(s)}$ under the conditions mentioned in the lemma.}  Fix $s \geq 1$. It is enough to show that, if $(\varphi^{(s)})'' (\beta^{(s)}) \neq  0$, then
\begin{align}\label{435}
	\limsup_{\delta\to 0}\;\beta^{(s+\delta)}\le\beta^{(s)}\le\liminf_{\delta\to 0}\;\beta^{(s+\delta)}.
\end{align}
The first inequality is obvious when $\beta^{(s)}=\iy$. If $\beta^{(s)} < \infty$ and $(\varphi^{(s)})''(\beta^{(s)}) \ne 0$, we necessarily have $(\varphi^{(s)})'' (\beta^{(s)}) < 0$.  Otherwise, if $(\varphi^{(s)})'' (\beta^{(s)}) > 0$, then \eqref{ac04} implies $(\varphi^{(s)})' (\beta^{(s)} - \nu ) < 1$ for sufficiently small $\nu>0$, a contradiction to the definition of $\beta^{(s)}$. 

From $(\varphi^{(s)})'' (\beta^{(s)}) < 0$, we deduce that, for sufficiently small $\nu>0$, $(\varphi^{(s)} )' (\beta^{(s)} + \nu ) < 1$. From the continuity of $s \mapsto (\varphi^{(s)})'$, we obtain that, for every $\delta_2$ with sufficiently small absolute value, $(\varphi^{(s+\delta_2)})' (\beta^{(s)} + \nu) < 1$. Therefore, $\beta^{(s+\delta_2)}<\beta^{(s)}+\nu$ and $\limsup_{\delta\to 0}\;\beta^{(s+\delta)}\le\beta^{(s)}+\nu$.  Because $\nu>0$ can be arbitrarily small, we get the first inequality in \eqref{435}.

We now turn to proving the second inequality in \eqref{435}.  Let $\gamma_1>0$ be arbitrary, and set $\hat\beta^{(s)} := \liminf_{\delta\to 0} \; \beta^{(s+\delta)}$.  If $\hat\beta^{(s)} = \iy$, then the second inequality in \eqref{435} is immediate, so we consider the case when $\hat\beta^{(s)} < \infty$. Let $\{\delta_j\}_{j \in \N}$ be a sequence such that $\lim_{j \to \infty} \del_j = 0$ and $\lim_{j \to \infty }\beta^{(s+\delta_j)} = \hat\beta^{(s)}$, with $\beta^{(s+\delta_j)}< \infty$ for every $j \in \N$.  Because $\varphi^{(s)}\in\calC^2(\R_+)$, we deduce that, for sufficiently large $j$,
\begin{align}\label{436}
	\left| (\varphi^{(s)} )' (\beta^{(s+\delta_j)}) - (\varphi^{(s)})' (\hat \beta^{(s)} ) \right| < \gamma_1.
\end{align}
Also, the continuity of  $s\mapsto (\varphi^{(s)})'$ implies that, for sufficiently large $j$,
\begin{align}\label{437}
	\left| (\varphi^{(s+\delta_j)})' (\beta^{(s+\delta_j)}) - (\varphi^{(s)})' ( \beta^{(s+\delta_j)} ) \right|<\gamma_1.
\end{align}
Recall $\beta^{(s+\delta_j)}< \infty$; thus, from \eqref{ac04}, we know $(\varphi^{(s+\delta_j)})' (\beta^{(s+\delta_j)}) = 1$.  From \eqref{436}--\eqref{437}, we obtain 
\begin{align}\notag
	\left| (\varphi^{(s)})' (\hat \beta^{(s)}) - 1\right| < 2\gamma_1.
\end{align}
Because $\gamma_1>0$ can be arbitrarily small, we get $(\varphi^{(s)})' (\hat \beta^{(s)}) = 1$, which implies $\beta^{(s)} \le \hat\beta^{(s)}$.
\end{proof}

\begin{remark}
Observe that, in the proof of \eqref{435}, the condition $(\varphi^{(s)})'' (\beta^{(s)}) \neq  0$ is only used to show the first inequality. The second inequality holds even when $(\varphi^{(s)})'' (\beta^{(s)}) =  0$.   \qed
\end{remark}

\begin{proof}[Proof of Proposition \ref{prop_43}] 
We break the proof into two parts. In the first part, we show existence of a solution to \eqref{HJB}; in the second part, we prove existence of uniform bounds for the derivative of this solution and the corresponding parameter $\beta_{\kappa}$. Note that once existence of a solution is shown, the verification result provided by Propositions~\ref{prop_41} and \ref{prop_42} proves that this solution equals the value function given by \eqref{eq:value}. Thus, we obtain uniqueness.

\vspace{0.2in}
\noindent
\emph{Existence.} 
We now construct a $\calC^2(\R_+)$ solution of \eqref{eq:phi_beta} for some $\beta_\kappa\in \R_+$ that also solves \eqref{HJB}.
To solve \eqref{eq:phi_beta}, we consider the Cauchy problem given in \eqref{eq:g^(s)}. Because below we find a solution $\varphi$ of \eqref{eq:g^(s)} that satisfies $1 \le \varphi' \leq \bar{s}$, with $\bar{s} = \frac{2m}{\sigma^2 \kappa}$, it also solves the same ordinary differential equation in \eqref{eq:g^(s)} with $H$ replacing $H_F$.

The outline of the proof is as follows.  First, we prove the existence of $s_\kap\in [1, \bar{s}]$ for which the parameter $\beta^{(s_\kap)}\in[0,\infty)$  in \eqref{ac09} satisfies the following conditions:
\begin{align}\label{ac01}
	\text{$(\varphi^{(s_\kap)})'(x) > 1\;$ on $\;[0,\beta^{(s_\kap)})$},
\end{align}
\begin{align}\label{ac01b}
	\text{$\quad (\varphi^{(s_\kap)})'(\beta^{(s_\kap)})=1$},
\end{align}
and  
\begin{align}\label{eq:k^s''beta-s}
	(\varphi^{(s_\kap)})''(\beta^{(s_\kap)})=0, \quad \hbox{if} \quad \beta^{(s_\kap)} > 0.
\end{align} 
Note that \eqref{eq:k^s''beta-s} implies 
$$
H(\beta^{(s_\kap)}, \varphi^{(s_\kap)}(\beta^{(s_\kap)}), {(\varphi^{(s_\kap)})}'(\beta^{(s_\kap)})) = \dfrac{2}{\sigma^2} \lt( m-\dfrac{\sigma^2 \kappa}{2} -\varrho \varphi^{(s_\kap)}(\beta^{(s_\kap)}) + f(\beta^{(s_\kap)}) \rt) = 0.
$$
Then, we define the function $\varphi$ by
\begin{equation}\label{eq:k-concat}
	\varphi(x)=\begin{cases}
		\varphi^{(s_\kap)}(x),\qquad &0\le x< \beta^{(s_\kap)},\\
		\varphi^{(s_\kap)}(\beta^{(s_\kap)})+(x-\beta^{(s_\kap)}),\qquad &\beta^{(s_\kap)}\le x < \infty,
	\end{cases}
\end{equation}
and note that $\varphi \in \calC^2(\R_+)$ solves \eqref{HJB}.  Indeed, for $0 \le x \le \beta^{(s_\kap)}$, $\varphi = \varphi^{(s_\kap)}$ solves the ODE in \eqref{eq:phi_beta} with $\varphi' \ge 1$.  Moreover, for $x> \beta^{(s_\kap)}$, we have
\begin{align*}
&\varphi''(x) + H(x, \varphi(x),\varphi'(x)) 
\\
&\quad
= 0 + \dfrac{2}{\sigma^2} \lt( m-\dfrac{\sigma^2 \kappa}{2} -\varrho \varphi(x) + f(x) \rt) \\
&\quad= \dfrac{2}{\sigma^2} \lt( m-\dfrac{\sigma^2 \kappa}{2} -\varrho \varphi^{(s_\kap)}(\beta^{(s_\kap)}) + f(\beta^{(s_\kap)}) \rt) - \dfrac{2}{\sigma^2} \lt( \varrho (x- \beta^{(s_\kap)}) - (f(x)-f(\beta^{(s_\kap)})) \rt)\\
&\quad=0-\dfrac{2}{\sigma^2} \lt( \varrho (x- \beta^{(s_\kap)}) - (f(x)-f(\beta^{(s_\kap)})) \rt) \leq 0,
\end{align*}
in which the last inequality follows by the Lipschitz continuity of $f$ in Assumption~\ref{as:f-lips}.  Finally, the proof of the existence part is complete by setting 
\begin{align}\label{424a}
	\beta_\kappa:=\beta^{(s_\kap)}, \qquad\text{and}\qquad g_{\beta_\kappa} := \varphi,  
\end{align}
in which $\varphi$ is given in \eqref{eq:k-concat}.\footnote{Recall that, in the proof of Theorem \ref{thm_41}, we use $g_{\beta_\kappa}$ to denote the $\calC^2(\R_+)$ solution of \eqref{HJB}.  In the rest of the proof, we write $g_{\beta_\kappa}$ instead of $\varphi$ to highlight the solution's dependence on $\kap$ via $\beta_\kap$.}  As a conclusion, we get 
\begin{align}\label{424aa}
	s_\kap = (g_{\beta_\kappa})'(0).
\end{align}

The rest of the proof in this part is, therefore, dedicated to showing the existence of $s_\kap$ that satisfies \eqref{ac01}--\eqref{eq:k^s''beta-s}.  To begin, look at \eqref{eq:g^(s)}, which provides part of the solution, namely, the part on the interval $[0, \beta_{\kappa})$. We have
\begin{equation}\label{eq:dd-0}
(\varphi^{(s)})''(0) = -\dfrac{2}{\sigma^2} \lt( m(\varphi^{(s)})'(0) - \dfrac{\sigma^2 \kappa}{2}((\varphi^{(s)})'(0))^2 - \varrho \varphi^{(s)}(0)+f(0) \rt) = s \left( \kappa s -\dfrac{2m}{\sigma^2} \right),
\end{equation}
in which we recall $f(0)=0$ from Assumption~\ref{as:f-lips}. By looking at the larger zero in $s$ on the right side of \eqref{eq:dd-0}, we consider three different cases for the parameters of the problem. 

\vspace{5pt}\noindent
{\it Case 1: $2m=\sigma^2 \kappa$.} In this case, the choice of $s_\kap=1$ provides $(\varphi^{(s_\kap)})'(0)=1$ and $(\varphi^{(s_\kap)})''(0)=0$. Consequently, $\beta^{(s_\kap)}=0$, and we have a solution. 

\vspace{5pt}
\noindent
{\it Case 2: $2m<\sigma^2 \kappa$.}  In this case, the choice of $s_\kap=1$ provides $(\varphi^{(s_\kap)})'(0)=1$ but $(\varphi^{(s_\kap)})''(\beta^{(s_\kap)}) < 0$. But, still the solution given by \eqref{eq:k-concat} works with $\beta^{(s_\kap)}=0$ because the condition \eqref{eq:k^s''beta-s} is required only to ensure smoothness while stitching the two parts in \eqref{eq:k-concat}.

\vspace{5pt}\noindent
{\it Case 3: $2m>\sigma^2 \kappa$.} 
We break our proof for this case into two steps: (1) showing the existence of a sufficiently small $s_1>1$ for which $(\varphi^{(s_1)})''(\beta^{(s_1)})<0$; (2) showing the existence of $s_\kap$ such that \eqref{ac01}, \eqref{ac01b} and \eqref{eq:k^s''beta-s} hold.

\vspace{0.1in}
\noindent
\emph{Step 1: }  
Set $1 \le s \le \bar{s} = \frac{2m}{\sigma^2 \kappa}$. For $0 \le x \le \beta^{(s)}$ (which implies $1 \le (\varphi^{(s)})'(x) \le s \le \bar{s}$), we have
\begin{equation}\label{eq:k^s''}
(\varphi^{(s)})''(x) = \kappa ((\varphi^{(s)})'(x))^2 -\frac{2m}{\sigma^2} (\varphi^{(s)})'(x)  + \frac{2}{\sigma^2} \lt( \varrho \varphi^{(s)}(x) - f(x) \rt).
\end{equation}
Then, note that $\bar s_1 \in (1, \bar{s})$, in which
\begin{equation}\label{eq:bar_s1}
 \bar s_1 := \inf\left\{s> 1: 8\varrho s(s-1) > \sigma^2 \lt[  s^2 \lt( \dfrac{2m}{\sigma^2} - \kappa s \rt)^2 \wedge \lt( \dfrac{2m}{\sigma^2} - \kappa \rt)^2 \rt] \right\}.
\end{equation}
Fix an arbitrary $s_1 \in (1,\bar s_1)$, and set
\begin{equation}\label{eq:MN}
M := \min\left\{ s_1 \left(\frac{2m}{\sigma^2} - \kappa s_1 \right), \frac{2m}{\sigma^2}- \kappa \right\},\qquad\text{and}\qquad N:=\frac{2(s_1-1)}{M}.
\end{equation}
Note that $M > 0$ and $N > 0$.  On the interval $[0, N]$, $(\varphi^{(s_1)})' \leq s_1$. Indeed, if that is not the case, we have $y_{s_1} = \inf \{ x \geq 0: (\varphi^{(s_1)})'(x) > s_1 \} \leq N$. By continuity of $(\varphi^{(s_1)})'$ and the definition of $y_{s_1}$, we have $(\varphi^{(s)})'(y_{s_1}) = s_1$ and $(\varphi^{(s_1)})''(y_{s_1}) > 0$. However, from \eqref{eq:k^s''}, \eqref{eq:bar_s1}, and \eqref{eq:MN}, we compute
$$
(\varphi^{(s_1)})''(y_{s_1}) = \kappa s_1^2 - \frac{2m}{\sigma^2}s_1 + \dfrac{2}{\sigma^2 } ( \varrho s_1 y_{s_1} - f(y_{s_1})) \leq - M + \frac{2}{\sigma^2 } \varrho s_1 y_{s_1} \le -\dfrac{ M}{2} < 0,
$$
a contradiction.  Consequently, we have $(\varphi^{(s_1)})' \leq s_1 $ on $[0, N]$. Furthermore, recall that $(\varphi^{(s_1)})'(x) \geq 1$ for $x \in [0, \beta^{(s_1)}]$. These two bounds on $(\varphi^{(s_1)})'$ allow us to bound $(\varphi^{(s_1)})''$ on $[0,N\wedge \beta^{(s_1)}]$. Specifically,
\begin{equation}\label{eq:dd-ub}
(\varphi^{(s_1)})''(x) \leq - M + \dfrac{2}{\sigma^2} \varrho s_1x \le -\dfrac{M}{2}, \qquad \text{ for all } x\in [0, N\wedge \beta^{(s_1)}].
\end{equation}
By integrating both sides of inequality \eqref{eq:dd-ub}, and by using $(\varphi^{(s_1)})'(0)=s_1$, we obtain
$$
(\varphi^{(s_1)})'(x) \leq s_1 - \dfrac{ M}{2} x, \qquad \text{ for all }x \in [0, N \wedge \beta^{(s_1)}].
$$
If $\beta^{(s_1)}>N$, then the inequality above holds for $x=N$, that is, $(\varphi^{(s_1)})'(N)\le1$. From the definition of $\beta^{(s_1)}$, we, then, deduce that $\beta^{(s_1)} \le N$, a contradiction to $\beta^{(s_1)}>N$.  This contradiction implies that we must have $\beta^{(s_1)}\le N$, so from \eqref{eq:dd-ub},
$$
(\varphi^{(s_1)})''(\beta^{(s_1)}) \le -\dfrac{ M}{2}<0,
$$
as we wished to show.

\vspace{0.1in}

\emph{Step 2:} In this step, we show the existence of $s_\kap \in (1, \bar{s}]$ that satisfies  \eqref{ac01}--\eqref{eq:k^s''beta-s}.
We define
\begin{equation}\label{eq:s_star}
s_\kap := \sup \left\{ s \in \left(s_1,\bar{s} \right): \forall s_1 < u < s,~ (\varphi^{(u)})''(\beta^{(u)}) < 0 ~ \text{ and }~\beta^{(u)} < \infty  \right\}.
\end{equation}
Observe that $s_\kap$ is potentially infinite when $\kappa = 0$. Assume first that $s_\kap < \infty$. We will show that $\beta^{(s_\kap)} < \infty$. If not, then we have $\beta^{(s_\kap)} = \infty$, and Lemma~\ref{lem_42} implies $\lim_{s \to s_\kap}\beta^{(s)} = \infty$. 
Consequently,
$$
\varrho \varphi(\beta^{(s)}) - f(\beta^{(s)}) \geq \varrho \beta^{(s)} - f(\beta^{(s)}) \geq  \delta \beta^{(s)} \to \infty, \text{ as } s \to s_\kap,
$$
in which $\varrho - \del$ is the Lipschitz constant for $f$ in Assumption \ref{as:f-lips}.  By applying this limit to \eqref{eq:k^s''}, we obtain
$$
(\varphi^{(s)})''(\beta^{(s)}) = \kappa \left( 1-\frac{2m}{\sigma^2 \kappa} \right) + \dfrac{2}{\sigma^2} \left( \varrho \varphi^{(s)} (\beta^{(s)}) - f(\beta^{(s)}) \right) \to \infty, \text{ as } s \to s_\kap,
$$
which contradicts $(\varphi^{(s)})''(\beta^{(s)}) \leq 0$. Thus, we have $\beta^{(s_\kap)} < \infty$. 

When $\kap = 0$, we now show that it is not possible for $s_\kap$ in \eqref{eq:s_star} to be infinite.  If $s_\kap = \infty$, we must have $\beta^{(s)} < \infty$ for all $s > s_1$. By following the same logic as in the first paragraph of Step 2, we deduce $\lim_{s \to \infty}\beta^{(s)} < \infty$, which implies that $\beta^{(s)}$ is uniformly bounded as $s \to \infty$.  This uniform bound implies that $f(\beta^{(s)})$ is also uniformly bounded.  By substituting $x=\beta^{(s)}$ in \eqref{eq:k^s''} and using $(\varphi^{(s)})''(\beta^{(s)}) \leq 0$ and $(\varphi^{(s)})'(\beta^{(s)})=1$, we obtain that $\varphi^{(s)}(\beta^{(s)})$ is also uniformly bounded.  Assumption~\ref{as:f-lips} and equation \eqref{eq:k^s''} with $\kap = 0$ imply, for $0 \le x \le \beta^{(s)}$,
\begin{align}\label{ineq:varphi2}
\begin{split}
(\varphi^{(s)})''(x) &=  - \, \dfrac{2m}{\sigma^2} (\varphi^{(s)})'(x) + \dfrac{2}{\sigma^2} (\varrho \varphi^{(s)}(x) - f(x))  \\
&\geq - \, \dfrac{2m}{\sigma^2} (\varphi^{(s)})'(x) + \dfrac{2}{\sigma^2} \delta x.
\end{split}
\end{align}
By integrating both sides of the inequality from $x = 0$ to $x = \beta^{(s)}$, we get
$$
(\varphi^{(s)})'(\beta^{(s)}) \geq s - \dfrac{2m}{\sigma^2 } \varphi^{(s)}(\beta^{(s)}) + \dfrac{\delta (\beta^{(s)})^2}{\sigma^2} \to \infty, \text{ as } s \to \infty,
$$
which contradicts $(\varphi^{(s)})''(\beta^{(s)}) \leq 0$.  Thus, when $\kap = 0$, then $s_\kap$ in \eqref{eq:s_star} is finite.

Observe that, because $\beta^{(s_\kap)} < \infty$, \eqref{ac04} implies $(\varphi^{(s_\kap)})'(\beta^{(s_\kap)}) =1$.   Next, we show that the smooth pasting boundary condition \eqref{eq:k^s''beta-s} is satisfied for the above choice of $s_\kap$, that is, $(\varphi^{(s_\kap)})''(\beta^{(s_\kap)}) = 0$. If not, then $(\varphi^{(s_\kap)})''(\beta^{(s_\kap)}) \neq 0$, and by Lemma~\ref{lem_42}, we have continuity of the mapping $s \mapsto (\beta^{(s)}, (\varphi^{(s)})''(\beta^{(s)}))$ at $s = s_\kap$.  In particular, if $(\varphi^{(s_\kap)})''(\beta^{(s_\kap)}) > 0$, there exists $\nu >0$ sufficiently small such that $\beta^{(s_\kap - \nu)} < \infty$ and $(\varphi^{(s_\kap - \nu)})''(\beta^{(s_\kap - \nu)}) > 0$, contradicting the definition of $s_\kap$ in \eqref{eq:s_star}.  On the other hand if $(\varphi^{(s_\kap)})''(\beta^{(s_\kap)}) < 0$, there exists $\nu_0 > 0$ sufficiently small such that $\beta^{(s_\kap + \nu)} < \infty$ and $(\varphi^{(s_\kap + \nu)})''(\beta^{(s_\kap + \nu)}) < 0$ for all $\nu \in [0, \nu_0)$, again contradicting the definition of $s_\kap$.

\vspace{0.2in}
\noindent
\emph{Boundedness of $\beta_{\kappa}$ and $(g_{\beta_\kappa})'$.}
We now show that the solution to \eqref{HJB} has bounded derivatives uniformly in $\kappa$. To that aim, we provide some additional insight about the solution $g_{\beta_{\kappa}}$ and $\varphi^{(s_\kap)}$, in which $s_\kap$ is given in \eqref{eq:s_star}; note that $s_\kap$ depends on $\kap$.  We have shown that for each $\kappa \geq 0$, $\beta_{\kappa} = \beta^{(s_\kap)}< \infty$. Consequently, because $(\varphi^{(s_\kap)})'(\beta^{(s_\kap)}) = 1$ and $(\varphi^{(s_\kap)})''(\beta^{(s_\kap)}) = 0$, Assumption~\ref{as:f-lips} and equation \eqref{eq:k^s''} imply
\begin{align}\label{eq:dd-lb}
0 &= \kappa - \dfrac{2m}{\sigma^2} + \dfrac{2}{\sigma^2} (\varrho \varphi^{(s_\kap)}(\beta^{(s_\kap)}) - f(\beta^{(s_\kap)}))\nonumber \\
&\geq \kappa - \frac{2m}{\sigma^2} + \frac{2}{\sigma^2} \delta \beta^{(s_\kap)}.
\end{align}
Thus,  $\beta_{\kappa} = \beta^{(s_\kap)} \leq \frac{2m-\sigma^2 \kappa}{2 \delta} \leq \frac{m}{\delta}$, uniformly in $\kappa \geq 0$. 

Therefore, for any $\kap \ge 0$, we can restrict our analyses of $\varphi^{(s_\kap)}$ or related ODEs to the interval $[0, m/\delta]$ as long as the coefficients are continuous, for which existence and uniqueness of solution follows from classical results.  Consider the case $2m > \sigma^2 \kappa$ because the other cases are trivial. 

We first find a uniform-in-$\kappa$ bound for $s_\kap$. Observe that, from \eqref{ineq:varphi2},
$$
(\varphi^{(s_\kap)})''(x) \geq -\dfrac{2m}{\sigma^2} (\varphi^{(s_\kap)})'(x),
$$ 
for $x \in [0, \beta_\kap]$.  Because $(\varphi^{(s_\kap)})'(x) \geq 1$ for $x \in [0, \beta_\kap]$, we have
$$
\dfrac{(\varphi^{(s_\kap)})''(x)}{(\varphi^{(s_\kap)})'(x)} \geq - \dfrac{2m}{\sigma^2 }.
$$
By integrating this inequality from $0$ to $x \in [0, \beta_\kappa] \subset [0, m/\delta]$, and by exponentiating the result, we obtain
$$
(\varphi^{(s_\kap)})'(x) \geq s_\kap \exp (-2mx/\sigma^2).
$$
If $s_\kap> \exp(2m^2/(\sigma^2 \delta))$, then $(\varphi^{(s_\kap)})'$ is never equal to $1$ on $[0, \beta_\kappa = \beta^{(s_\kap)}]$, which contradicts the fact that $(\varphi^{(s_\kap)})'(\beta^{(s_\kap)}) = 1$. Consequently, we have a uniform bound
\begin{equation}\label{eq:s_star_bnd}
s_\kap \le \exp(2m^2/(\sigma^2 \delta)) < 1 + \exp(2m^2/(\sigma^2 \delta)) =: \tilde{s}.
\end{equation}
for all $\kappa \geq 0$.

Next, observe that, because $\kappa < \frac{2m}{\sigma^2}$, then from \eqref{eq:k^s''}, we have
$$
(\varphi^{(s)})''(x) < \dfrac{2m}{\sigma^2} ((\varphi^{(s)})'(x))^2 + \dfrac{2 \varrho}{\sigma^2}  \varphi^{(s)}(x).
$$
for all $x \in [0, \beta^{(s)}]$ and for all $1 \le s < \tilde{s}$.  We claim that for any $s < \tilde{s}$, $(\varphi^{(s)})'(x) \le (\Psi^{(\tilde{s})})'(x)$, in which $\Psi^{(\tilde{s})} \in \calC^2(\R_+)$ uniquely solves
$$
(\Psi^{(\tilde{s})})''(x) = \dfrac{2m}{\sigma^2} ((\Psi^{(\tilde{s})})'(x))^2 + \dfrac{2 \varrho}{\sigma^2}  \Psi^{(\tilde{s})}, \qquad \Psi^{(\tilde{s})}(0)=0, \qquad (\Psi^{(\tilde{s})})'(0)=\tilde{s}.
$$
If not, let $s<\tilde s$ and $\hat x_s=\hat{x} = \inf \{ x > 0: (\varphi^{(s)})'(x) > (\Psi^{(\tilde s)})'(x)\}$, which implies $(\varphi^{(s)})'(x) \le (\Psi^{(\tilde s)})'(x)$ and, thus, $\varphi^{(s)}(x) \le \Psi^{(\tilde s)}(x)$ for all $x \in [0, \hat{x}]$.  Note that $\hat{x} > 0$ because $(\varphi^{(s)})'(0) = s < \tilde{s} = (\Psi^{(\tilde s)})'(0)$. We, then, obtain
\begin{align*}
(\varphi^{(s)})'(\hat{x}) - s =  \int_0^{\hat{x}} (\varphi^{(s)})''(u) du &< \int_0^{\hat x} \left( \dfrac{2m}{\sigma^2} ((\varphi^{(s)})'(u))^2 + \dfrac{2}{\sigma^2} \varrho \varphi^{(s)}(u) \right) du\\
& \le \int_0^{\hat x} \left( \dfrac{2m}{\sigma^2} ((\Psi^{(\tilde s)})'(u))^2 + \dfrac{2}{\sigma^2} \varrho \Psi^{(\tilde s)}(u) \right) du \\
&= \int_0^{\hat{x}} (\Psi^{(\tilde s)})''(u) du = (\Psi^{(\tilde s)})'(\hat{x}) - \tilde s,
\end{align*}
which implies
\[
(\varphi^{(s)})'(\hat{x}) < (\Psi^{(\tilde s)})'(\hat{x}) - (\tilde s - s) < (\Psi^{(\tilde s)})'(\hat{x}),
\]
thereby contradicting the definition of $\hat{x}$.  Thus, $(\varphi^{(s)})'(x) \le (\Psi^{(\tilde s)})'(x)$ for all $x \in [0, \beta^{(s)}]$ and all $s < \tilde s$.  Because $\sup_{x \in [0, m/\delta]} |(\Psi^{(\tilde s)})'(x)|$ is finite, we have a bound for $(\varphi^{(s_\kap)})'$ on $[0, \beta_{\kappa}]$, uniformly in $\kappa$.  We, thus, deduce that the solution $g_{\beta_{\kappa}}$ of \eqref{HJB}, for any $\kap \ge 0$, satisfies
\begin{equation}\label{eq:der-bnd}
(g_{\beta_{\kappa}})'(x) \leq \bar{c}, \qquad x \ge 0,
\end{equation}
in which the constant $\bar{c}$ is independent of $\kappa$.
\end{proof}

\begin{remark}\label{rem:F}
Because we have proved the existence of an upper bound for $(g_{\beta_{\kappa}})'$, uniform in $\kappa$, we can replace $\bar s = 2m/(\sig^2 \kap)$ in the definition of $H_F$ in \eqref{HF} with the constant $\bar{c}$.  \qed
\end{remark}

\begin{proof}[Proof of Proposition~\ref{prop_44}]
Recall the definition of $\beta_\kappa$ from Theorem~\ref{thm_41}. Define
\begin{align}\label{441}
	\hat\beta_\kappa&:=\sup\left\{x>0 : \forall y\le x,\;
	V''(y;\kappa)+H(y,V(y;\kappa),V'(y;\kappa))=0\right\}.
\end{align}
From Proposition~\ref{prop_43} along with the relation \eqref{eq:phi_beta} it follows that $\hat\beta_\kappa\ge\beta_\kappa$. Because on the interval $[\beta_\kappa,\hat\beta_\kappa]$, both of the conditions 
\begin{align}\label{442}
	V''(x;\kappa)+H(x,V(y;\kappa),V'(x;\kappa))=0\qquad
	\text{and}\qquad
	V'(x;\kappa)=1
\end{align} hold, it follows that $V$ solves \eqref{HJB} for any $\beta\in[\beta_\kappa,\hat\beta_\kappa]$. Propositions \ref{prop_41} and \ref{prop_42} imply that every such $\beta$-threshold strategy is optimal. 
Thus, the non-uniqueness of the optimal threshold strategy is equivalent to the existence of non-degenerate interval $[\beta_\kappa,\hat\beta_\kappa]$, on which the equations in \eqref{442} hold. By combining them, we get $V(x;\kappa)=(m-\sigma^2 \kappa/2+f(x))/\varrho$. By using again $V'(x;\kappa)=1$, we deduce that $f' \equiv \varrho$ on $[\beta_\kappa,\hat\beta_\kappa]$, which contradicts Assumption~\ref{as:f-lips}, namely, that $f$ is Lipschitz continuous with the Lipschitz constant strictly smaller than $\varrho$.  We have, thereby, shown uniqueness of the optimal threshold strategy. 
\end{proof}

\begin{proof}[Proof of Proposition~\ref{prop_45}]
From Propositions~\ref{prop_41}--\ref{prop_43}, we know that the value function $V(\cdot; \kappa)$ given by \eqref{eq:value} satisfies the conditions in \eqref{eq:phi_beta} for $\beta=\beta_{\kappa}$.  Thus, $V''(x;\kappa) = 0$ for $x \geq \beta_{\kappa}$, while for $x < \beta_{\kappa}$, we have
\begin{equation}\label{eq:45-1}
V''(x; \kappa) = \kappa (V'(x; \kappa))^2 - \dfrac{2m}{\sigma^2} V'(x;\kappa) + \dfrac{2}{\sigma^2} \left( \varrho V(x;\kappa) - f(x) \right).
\end{equation}
To derive a contradiction, assume there exists some $x_0 < \beta_{\kappa}$ such that $V''(x_0 ; \kappa) > 0$. Thus, $V'(x_1;\kappa) > V'(x_0; \kappa) > 1$ for some $x_1 > x_0$, sufficiently close to $x_0$. However, because $\beta_{\kappa} < \infty$ and $V'(\beta_{\kappa}; \kappa) = 1$, in order for $V'$ to penetrate the $y=V'(x_0;\kappa)$ barrier, there must exist $z \in (x_1, \beta_{\kappa})$ such that $V'(z; \kappa) = V'(x_0; \kappa)$ and $V''(z; \kappa) \leq 0$. From Assumption~\ref{as:f-lips}, note that $f(z) \leq f(x_0) + (\varrho - \delta)(z-x_0)$. Furthermore, since $V' \geq 1$, we have $V(z;\kappa) \geq V(x_0;\kappa) + (z-x_0)$. Thus, from \eqref{eq:45-1}, we get
\begin{align*}
V''(z; \kappa) &= \kappa (V'(z; \kappa))^2 - \dfrac{2m}{\sigma^2} V'(z;\kappa) + \dfrac{2}{\sigma^2} \left( \varrho V(z;\kappa) - f(z) \right)\\
&> \kappa (V'(x_0; \kappa))^2 - \dfrac{2m}{\sigma^2} V'(x_0;\kappa) + \dfrac{2}{\sigma^2} \left( \varrho V(x_0;\kappa) - f(x_0) \right) = V''(x_0;\kappa) > 0,
\end{align*}
which contradicts $V''(z; \kappa) \leq 0$.  Hence, $V''(x ; \kappa) \leq 0$ for all $x \in [0, \infty)$.
\end{proof}

\section{Optimal strategy and dependency on the ambiguity parameter}\label{sec_4}

In this section, we study the dependence of the value function $V$ and optimal threshold $\beta_\kap$ on the ambiguity parameter $\kap$. 
We show continuity of $V$ and $\beta_\kap$ on $\kap$, and we show that, as $\kappa \to 0^+$, our model converges to the risk-neutral model.  For the latter, recall the definition of $V(\cdot;0)$ given in \eqref{eq:V-risk-neutral}.

\begin{theorem}\label{thm_45}
	The mapping $[0,\iy)\ni\kappa\mapsto V(x; \kappa)$ is decreasing and continuous, uniformly in $x \in\R_+$. 
	Moreover,  there is a constant $C > 0$ such that for every $\kappa\in(0,\iy)$, $\sup_{x\in[0,\infty)}|V(x;\kappa)-V(x;0)|\le C \cdot \kappa$. Also, $\lim_{\kappa\to\iy}V(x;\kappa)=0$, uniformly in $x \in \R_+$. 
\end{theorem}

\begin{proof}
We start by showing the monotonicity and continuity of the mapping $(0,\iy)\ni\kappa\mapsto V(\cdot;\kappa)$. The proof for $\kappa=0$ is given separately. 
Fix $0<\kappa_1<\kappa_2$.  Let $\Q^i$ and $\xi_i = \{ \xi_{i,t} \}_{t \ge 0}$ denote $\Q^{V(\cdot;\kappa_i)}$ and $\xi^{V(\cdot;\kappa_i)} = \big\{ \xi^{V(\cdot;\kappa_i)}_{t} \big\}_{t \ge 0}$, respectively, for $i=1,2$. Then, for every $x > 0$, from \eqref{eq:equil}, we have
\begin{align*}
	V(x;\kappa_1)&=\sup_{D\in{\calA}(x)}J(x,D,\Q^1;\kappa_1) \\
	&= \sup_{D\in{\calA}(x)}\left[J(x,D,\Q^1;\kappa_2)+\frac{1}{2}\left(\frac{1}{\kappa_1}-\frac{1}{\kappa_2}\right) \int_0^{\tau} e^{-\varrho t}\E^{\Q^1}\big[(\xi_{1,t})^2\big] dt\right]
\end{align*}
Because $\kappa_2 > \kappa_1>0$ and $\xi_{1, t} = -\kappa_1 \sigma V'(X_t; \kappa_1) \leq - \kappa_1 \sigma < 0$, we deduce
\begin{equation}\label{eq:V-lb}
	V(x;\kappa_1) > \sup_{D\in{\calA}(x)} J(x,D,\Q^1;\kappa_2) \geq \sup_{D\in{\calA}(x)} \inf_{\Q \in \calQ(x)}J(x,D,\Q;\kappa_2) = V(x ; \kappa_2).
\end{equation}
We have, thus, shown $\kappa\mapsto V(\cdot;\kappa)$ is strictly decreasing on the interval $(0, \iy)$.  We argue monotonicity at $\kappa=0$ as follows: for $x > 0$ and $\kappa_1 > 0$,
\begin{align}\label{eq:mon-0}
	V(x;0)&=\sup_{D\in{\calA}(x)}J(x,D,\PP;\kappa_1)\notag \\
	&\ge\sup_{D\in{\calA}(x)}\inf_{\Q\in\hat\calQ(x)}J(x,D,\Q; \kappa_1)\notag \\
	&=V(x; \kappa_1).
\end{align}
Also, for $0 < \kap_1 < \kap_2$,
\begin{align*}
V(x ; \kappa_1) &= \inf_{\Q \in \calQ(x} J(x, D^1, \Q; \kap_1) \le J(x, D^1, \Q^2; \kap_1) \leq \sup_{D \in \calA(x)} J(x, D, \Q^2 ; \kappa_1)\\
&= \sup_{D \in {\calA}(x)} \left[ J(x, D, \Q^2; \kappa_2) + \dfrac{1}{2} \left( \dfrac{1}{\kappa_1} - \dfrac{1}{\kappa_2} \right) \int_0^{\tau} e^{-\varrho t} \E^{\Q^2} \big[ (\xi_{2,t})^2 \big]dt \right]
\end{align*}
Note, by using \eqref{eq:der-bnd}, $\xi_{2, t} = -\kappa_2 \sigma V'(X_t;\kappa_2) \geq -\kappa_2 \sigma \bar{c}$. This bound implies $(\xi_{2,t})^2 \leq \kappa_2^2 \sigma^2 \bar{c}^2$ and, consequently,
\begin{align}\label{eq:V-ub}
V(x; \kappa_1)  &\leq \sup_{D \in {\calA}(x)} \left[ J(x, D, \Q^2; \kappa_2) + \dfrac{1}{2}\left( \dfrac{1}{\kappa_1} - \dfrac{1}{\kappa_2} \right) \kappa_2^2 \sigma^2 \bar{c}^2 \dfrac{1-e^{-\varrho \tau}}{\varrho} \right] \nonumber \\ 
&= V(x;\kappa_2) + (\kappa_2 - \kappa_1) \dfrac{\kappa_2}{\kappa_1} \dfrac{\sigma^2 \bar{c}^2}{2\varrho}. 
\end{align}
By combining \eqref{eq:V-lb} and \eqref{eq:V-ub}, we obtain
\begin{align}\notag
	V(x;\kappa_2) < V(x;\kappa_1)\le V(x;\kappa_2)+\frac{\kappa_2\sigma^2 \bar{c}^2}{2\kappa_1\varrho}(\kappa_2-\kappa_1),
\end{align}
from which follows continuity of $\kappa\mapsto V(\cdot;\kappa)$ on the interval $(0,\iy)$, uniformly in $x$.

We next prove continuity of $\kappa\mapsto V(\cdot;\kappa)$ at $\kappa=0$. Note that in the arguments above, we cannot relax the inequality $\kappa_1>0$ to $\kappa_1 \ge 0$ because we divide by $\kappa_1$. Therefore, we use a different proof to show continuity at $\kap = 0$.  Let $\beta_0$ denote the optimal threshold for the risk-neutral problem.  Let $\kappa>0$, and let $\Q^\kappa$ and $\xi^\kap$ denote, respectively, $\Q^{V(\cdot;\kappa)}$ and $\xi^{V(\cdot;\kappa)}$.  Observe that, from \eqref{eq:mon-0},
\begin{equation}\label{eq:5-1-0}
V(x;0) \geq V(x;\kappa) = \sup_{D \in {\calA}(x)}J(x, D, \Q^{\kappa};\kappa) \geq J(x, D^{\beta_0}, \Q^{\kappa};\kappa).
\end{equation}
From \eqref{cost2} and \eqref{eq:Q},
\begin{equation}\label{eq:5-1-a}
J(x, D^{{\beta_0}}, \Q^{\kappa};\kappa) = \E^{\Q^{\kappa}} \left[ \int_0^{\tau} e^{-\varrho t} \left\{ f(X_t^{\beta_0})dt + dD_t^{\beta_0} \right\} \right] + \dfrac{\kappa \sigma^2}{2} \E^{\Q^{\kappa}} \left[ \int_0^{\tau} e^{-\varrho t} (V'(X_t ; \kappa))^2 dt \right].
\end{equation}
Because $1 \leq V' \leq \bar{c}$, the second term in the right side of \eqref{eq:5-1-a} vanishes as $\kappa \to 0^+$. Thus, owing to \eqref{eq:5-1-0} and \eqref{eq:5-1-a}, and by recalling the payoff function for the risk-neutral problem, to show continuity of $\kappa \mapsto V(\cdot;\kappa) $ at $\kap = 0$, it is enough to show
\begin{equation}\label{eq:rtp-1}
\lim_{\kappa \to 0^+} \E^{\Q^{\kappa}} \left[ \int_0^{\tau} e^{-\varrho t} \left\{ f(X_t^{\beta_0})dt + dD_t^{\beta_0} \right\} \right]  = \E^{\PP} \left[ \int_0^{\tau} e^{-\varrho t} \left\{ f(X_t^{{\beta_0}})dt + dD_t^{{\beta_0}} \right\} \right].
\end{equation}
To that end, we construct two coupled processes, one each for $\kappa=0$ and $\kappa>0$, on the probability space $(\Omega, \calF, \F, \check\PP)$ that supports a one-dimensional standard Brownian motion $B$.  By using Definition \ref{def_Skorokhod}, we construct the processes as follows: for $t \ge 0$,
\begin{align}\notag
	(X^\kappa_t, D^\kappa_t) &:=\Gamma_{{\beta_0}} \Big(x+m\cdot-\int_0^\cdot\kappa \sigma^2V'(X_s^\kappa;\kappa)ds+\sigma B_\cdot \Big)_t,\\\notag
	(X^0_t, D^0_t) &:=\Gamma_{{\beta_0}} \big(x+m\cdot+\sigma B_\cdot \big)_t.
\end{align}
Note that, by \eqref{eq:Q}, $(X^0,D^0)$ (resp., $(X^\kappa, D^{\kappa})$) has the same distribution under the measure $\check\PP$ as $(X^{{\beta_0}},D^{{\beta_0}})$ under the measure $\PP$ (resp., $\Q^\kappa$). Hence, \eqref{eq:rtp-1} is equivalent to
\begin{align}\label{457}
	\lim_{\kappa \to 0^+} \E^{ \check\PP}\left[ \int_0^\tau e^{-\varrho t} \Big\{ f(X_t^\kappa)dt + d D_t^\kappa \Big\} \right]=\E^{ \check\PP} \left[ \int_0^\tau e^{-\varrho t} \Big\{ f(X_t^0)dt + d D_t^0 \Big\} \right].
\end{align}
Now
\begin{align*}
&\left| \E^{ \check\PP}\left[ \int_0^\tau e^{-\varrho t} \Big\{ f(X_t^\kappa)dt + d D_t^\kappa \Big\} \right] -  \E^{ \check\PP} \left[ \int_0^\tau e^{-\varrho t} \Big\{ f(X_t^0)dt + d D_t^0 \Big\} \right] \right|\\
&= \left| \E^{ \check\PP} \left[ \int_0^{\tau} e^{-\varrho t} \big( f(X_t^{\kappa}) - f(X_t^0) \big)dt + \int_0^{\tau} e^{-\varrho t} \big(dD_t^{\kappa} - dD_t^0 \big) \right] \right|\\
&\leq \E^{\check \PP} \left[ \int_0^{\tau} e^{-\varrho t} \left| f(X_t^{\kappa})-f(X_t^0) \right|dt + \int_0^{\tau} \varrho e^{- \varrho t} \left|D_t^{\kappa} - D_t^0\right| dt + e^{-\varrho \tau} \left| D_{\tau}^{\kappa} - D_{\tau}^0 \right| \right]
\end{align*}
From the Lipschitz continuity of $f$ in Assumption~\ref{as:f-lips} and from Lemma~\ref{lem_Skorokhod}, we obtain
\begin{align*}
&\left| \E^{ \check\PP}\left[ \int_0^\tau e^{-\varrho t} \Big\{ f(X_t^\kappa)dt + d D_t^\kappa \Big\} \right] -  \E^{ \check\PP} \left[ \int_0^\tau e^{-\varrho t} \Big\{ f(X_t^0)dt + d D_t^0 \Big\} \right] \right|\\ 
&\leq \E^{\check \PP} \left[ \int_0^{\tau} \varrho e^{-\varrho t} \big( \left|X_t^{\kappa} - X_t^0\right| + \left|D_t^{\kappa} - D_t^0 \right|\big) dt + e^{-\varrho \tau} \left| D_{\tau}^{\kappa} - D_{\tau}^0 \right| \right]\\
&\leq \E^{\check \PP} \left[ \int_0^{\tau} \varrho e^{-\varrho t} c_S \sup_{s \in [0,t]} \left|  \kappa \sigma^2 \int_0^s V'(X_u^{\kappa};\kappa)du\right|dt + e^{-\varrho \tau} \sup_{s \in [0, \tau] } \left| \kappa \sigma^2 \int_0^s V'(X_u^{\kappa};\kappa)du \right| \right]\\
&\leq \E^{\check \PP} \left[ \kappa \sigma^2 c_S \bar{c} \int_0^{\tau} \varrho t e^{-\varrho t} dt + \tau e^{-\varrho \tau} \kappa \sigma^2 \bar{c} \right] \leq \dfrac{\kappa \bar{c} \sigma^2}{\varrho} (c_S + e^{-1}) \to 0,
\end{align*}
as $\kappa \to 0^+$, as required.

We now turn to proving the limit $\lim_{\kappa\to\iy}V(\cdot;\kappa)=0$. To show this limit, given $\kap > 0$, consider the probability measure $\Q^\kappa$ that is associated with the strategy $\xi_t^\kappa \equiv -\kappa^{1/4}$, for $t \ge 0$.  Under $\xi^\kap$, for any $D \in \calA(x)$, $X$ follows the dynamics
\begin{align}\label{ac11}
	X_t= x+ (m-\sigma\kappa^{3/4})t+ \sigma  B_t^{\Q^\kappa} - D_t =: Z_t - D_t,
\end{align}
for $ t\ge 0$, in which $ B^{ \Q^\kappa}$ is the $\F$-standard Brownian motion under $\Q^\kappa$ given by $B^{ \Q^\kappa}_t = B_t + \kap^{1/4} t$. The payoff function is given by
\begin{align}\label{eq:payoff_0}       
	J(x, D, \Q^\kappa;\kappa)&=
	\E^{ \Q^{\kappa}}\Big[\int_0^{\tau} e^{-\varrho t}\Big( f( X_t)dt +  d D_t + \frac{1}{2\kappa} (\xi_t^{\kappa})^2 dt\Big) \Big] \notag \\
	&\leq \E^{ \Q^{\kappa}}\Big[\int_0^{\tau} e^{-\varrho t}\Big( f( X_t)dt +  d D_t \Big) \Big] + \frac{1}{2\varrho\sqrt{\kappa}} \notag \\
	&\le \E^{ \Q^{\kappa}}\left[\int_0^{\tau} e^{-\varrho t}\Big( f( Z_t) + \varrho Z_t \Big) dt  + e^{-\varrho \tau}  D_{\tau} \right] + \frac{1}{2\varrho\sqrt{\kappa}},
\end{align}
in which the second inequality follows from integration by parts and $Z_t \ge \min(X_t, D_t)$ for all $t \ge 0$.  Furthermore, one can check that $\tau^* \to 0^+$ $\Q^{\kappa}$-a.s.\ as $\kappa \to \infty$, in which $\tau^* := \inf\{t >0: Z_t = 0\} \ge \tau$. Consequently, $\tau \to 0$ $\Q^{\kappa}$-a.s.\ as $\kappa \to \infty$.  By right continuity of $D$, we also have $D_{\tau} \to 0$ $\Q^{\kappa}$-a.s.\ as $\kappa \to \infty$.  By using these limits, we see that the expectation in \eqref{eq:payoff_0} vanishes as $\kappa \to \infty$. We, thus, have for any $D \in {\mathcal{A}}(x)$,
$$
\lim_{\kappa \to \infty} J(x, D, \Q^{\kappa}; \kappa) = 0,
$$
and our result follows.
\end{proof}

In the following theorem, we show that the minimizer's optimal control is continuous with respect to $\kappa$, which will follow readily from \eqref{eq:Q} and from showing that the map $\kappa \mapsto V'(\cdot; \kappa)$ is continuous with respect to $\kap$.  We also show continuity of the threshold $\beta_\kappa$ with respect to $\kap$; note that we do not have an explicit expression for $\beta_\kap$.

\begin{theorem}\label{thm_46}
The mapping $[0,\iy)\ni \kappa \mapsto V'(x; \kappa)$ is continuous, uniformly in $x \in \R_+$. In addition, for any given $\kappa \ge 0$, 
	\begin{align}\label{460}
		\beta_\kappa = \lim_{\delta\to 0}\beta_{\kappa+\delta}.
	\end{align}
\end{theorem}

\begin{proof}
For ease of reading, we present the proof in four steps.

\vspace{0.1in}

\emph{Step 1: Continuity of $\kappa \mapsto V'(0;\kappa)$.}
Observe that, because $V(0;\kappa) = 0$ and $\kappa \mapsto V(x;\kappa)$ is decreasing (Theorem~\ref{thm_45}), we deduce that, if $\kappa_1 < \kappa_2$, then $V'(0;\kappa_1) \geq V'(0;\kappa_2)$.  Suppose $\kappa \mapsto V'(0;\kappa)$ is not continuous. Then, there exists $ 0 \leq \tilde \kappa < \frac{2m}{\sigma^2}$ (recall that if $\kappa \geq \frac{2m}{\sigma^2}$, then $\beta_{\kappa}=0$ with $V' \equiv 1$ on $\R_+$) such that
$$
a_1 := \liminf_{\ep \to 0^+} V'(0;\tilde \kappa - \ep) > \limsup_{\ep \to 0^+} V'(0; \tilde \kappa + \ep) =: a_2 \geq 1,
$$
since $V'(0; \kap) \geq 1$. It follows that, for all $0< \ep < \tilde \kappa \wedge (\frac{2m}{\sigma^2} - \tilde \kappa)$,
$$
V'(0; \tilde \kappa - \ep) > V'(0;\tilde \kappa + \ep) + \dfrac{a_1 -a_2}{2}.
$$
Because $V'(\cdot;\kappa) \leq \bar c$ (recall \ref{eq:der-bnd}) and $\beta_{\kappa} \leq \frac{m}{\delta}$ are uniformly bounded in $\kappa$, from \eqref{HJB}, \eqref{eq:H}, and \eqref{eq:phi_beta}, we get that $V''(\cdot;\kappa)$ is also uniformly bounded in $\kappa$.  From $V''(\cdot;\kappa)$'s uniform bound and the fact that $V(\cdot;\kappa) \in \mathcal{C}^2(\R_+)$, we now obtain that there exists $\bar \delta > 0$, independent of $\kappa$, such that for all $0< \ep < \tilde \kappa \wedge (\frac{2m}{\sigma^2} - \tilde \kappa)$,
$$
V'(x;\tilde \kappa - \ep) > V'(x; \tilde \kappa + \ep) + \dfrac{a_1 -a_2}{2}, \qquad \text{for all } x \in [0, \bar \delta].
$$
By integrating in $x = 0$ to $\bar \delta$, we obtain
$$
V(\bar \delta; \tilde \kappa - \ep) > V( \bar \delta ; \tilde \kappa + \ep) + \dfrac{a_1 - a_2}{2} \, \bar{\delta},
$$
for all $0< \ep < \tilde \kappa \wedge (\frac{2m}{\sigma^2} - \tilde \kappa)$.  Finally, by taking limit $\ep \to 0^+$, the continuity of $\kappa \mapsto V(\cdot; \kappa)$ implies
$$
V(\bar \delta; \tilde \kappa) \geq V(\bar \delta; \tilde \kappa) + \dfrac{a_1 -a_2}{2}  \, \bar{\delta},
$$
contradicting the continuity of $\kappa \mapsto V(x; \kappa)$.  Thus, we have shown $\kappa \mapsto V'(0;\kappa)$ is continuous.

\vspace{0.1in}

\emph{Step 2: Continuity of $\kappa \mapsto(l^{(\kappa)}(\cdot),(l^{(\kappa)})'(\cdot))$.}
Recall the definition of $\varphi^{(s)}$ given in \eqref{eq:g^(s)}, and recall that $\beta_{\kappa} \leq \frac{m}{\delta}$.  Let $l^{(\kappa)}(x)$ denote $\varphi^{(V'(0;\kappa))}(x)$ for $x\in[0,\frac{m}{\delta}]$. More explicitly,
\begin{align}\label{462}
	\begin{cases}
		(l^{(\kappa)})''(x)+H_F(x,l^{(\kappa)}(x),(l^{(\kappa)})'(x))=0,\qquad x\in[0,\frac{m}{\delta}],\\
		(l^{(\kappa)})(0)=0,\quad (l^{(\kappa)})'(0)=V'(0;\kappa),
	\end{cases}
\end{align}
with $\bar s = 2m/(\sig^2 \kap)$ replaced by the constant $\bar{c}$ in $F$, in conjunction with Remark~\ref{rem:F}.  For any $\kappa \ge 0$, the function $V(\cdot;\kappa)$ satisfies \eqref{eq:phi_beta} and, therefore, also \eqref{462} on $[0,\beta_{\kappa}]$. Uniqueness of the solution implies that 
\begin{align}\label{467}
	V(x;\kappa)=l^{(\kappa)}(x), \quad V'(x;\kappa)=(l^{(\kappa)})'(x),  \qquad x\in[0,\beta_{\kappa}].
\end{align}

We now show that the mapping $[0,\iy)\ni\kappa\mapsto(l^{(\kappa)}(x),(l^{(\kappa)})'(x))$ is continuous, uniformly in $x \in [0, \frac{m}{\delta}]$. Fix $\kappa_1,\kappa_2\in\R_+$. For simplicity of notation, let $f_i$ denote $l^{(\kappa_i)}$ for $i = 1, 2$; then,
\begin{align}\notag\label{463}
	f_1'(x)&=V'(0;\kappa_1)-\int_0^xH^{\kappa_1}_F(y,f_1(y),f_1'(y))dy, \\\notag
	f_2'(x)&=V'(0;\kappa_2)- \int_0^xH^{\kappa_2}_F(y,f_2(y),f_2'(y))dy \\
	&=V'(0;\kappa_2)  -\int_0^xH^{\kappa_1}_F(y,f_2(y),f_2'(y))dy+\int_0^x(\kappa_2-\kappa_1)F^2(f_2'(y)) dy,
\end{align}
in which the superscript $\kappa_i$ in $H^{\kappa_i}_F$ emphasizes $H_F$'s dependence on $\kappa_i$, for $i=1,2$.  Also, from \eqref{eq:F}, we know $|F(f_2'(y))| \le 2\bar{c}$.  From \eqref{ac05} and the expressions in \eqref{463}, it follows that there exists a constant $\tilde L>0$, independent of $\kappa_1$ and $\kappa_2$, such that 
\begin{align}\label{464}
	\big| f_1'(x)-f_2'(x) \big| &\le  \tilde L\left(\int_0^x\Big[ \big|f_1(y)-f_2(y) \big| + \big| f_1'(y)-f_2'(y) \big|\Big]dy +|\kappa_1-\kappa_2|\right) \notag \\
	&\quad + \big|V'(0;\kappa_1) - V'(0;\kappa_2) \big|.
\end{align}
Also, because $f_i(0)=0$ for $i = 1, 2$, we have,
\begin{align}\notag
	f_1(x)-f_2(x) = \int_0^x \big( f_1'(y)-f_2'(y) \big) dy,
\end{align}
which implies
\begin{equation}\label{464a}
|f_1(x)-f_2(x)|\le \int_0^x \big| f_1'(y)-f_2'(y) \big| dy.
\end{equation}
From inequalities \eqref{464} and \eqref{464a} and from Gr\"onwall's inequality, we deduce there is a constant $C>0$, independent of $\kappa_1$ and $\kappa_2$, such that
\begin{align}\notag
	\sup_{x \in [0, \frac{m}{\delta}]} \left( \big|f_1(x)-f_2(x)\big| + \big|f_1'(x)-f_2'(x)\big| \right)\le C\left( \big|V'(0;\kappa_1)-V'(0;\kappa_2)\big| + |\kappa_1-\kappa_2|\right).
\end{align}
Recalling from Step 1 that $\kappa\mapsto V'(\cdot;\kappa)$ is continuous, we obtain
\begin{equation}\label{eq:cont}
\lim_{|\kappa_1 - \kappa_2| \to 0^+}\sup_{x \in [0, \frac{m}{\delta}]} \left( \big|f_1(x)-f_2(x) \big| + \big|f_1'(x)-f_2'(x)\big| \right) = 0.
\end{equation}
We have, thus, shown that the mapping $[0,\iy)\ni\kappa\mapsto(l^{(\kappa)}(x),(l^{(\kappa)})'(x))$ is continuous, uniformly in $x \in [0, \frac{m}{\delta}]$. 

\vspace{0.1in}

\emph{Step 3: Continuity of $\kappa \mapsto \beta_{\kappa}$.}
Note that, from \eqref{467}, the definition of $\beta_\kappa$, the definition of $\beta^{(s)}$ in \eqref{ac09}, and recalling that $l^{(\kappa)}=\varphi^{(V'(0;\kappa))}$, it follows that for any $\kappa\in[0,\iy)$, $\beta_{\kappa}=\beta^{(V'(0;\kappa))}$. From the continuity of $\kappa\mapsto V'(0;\kappa)$ obtained in Step 1 and from the second inequality in \eqref{435}, we get
\begin{equation}\label{eq:beta-1}
\beta_{\kappa} \leq \liminf_{\ep \to 0} \beta_{\kappa+\ep}.
\end{equation}
From Theorem~\ref{thm_45} and the above considerations, we have
$$
\lim_{\ep \to 0} V(\cdot; \kappa+\ep) = V(\cdot; \kappa) \quad  \text{and} \quad \lim_{\ep \to 0} l^{(\kappa + \ep)}(\cdot) = l^{(\kappa)}(\cdot),
$$
uniformly on $[0, \frac{m}{\delta}]$.  In addition, \eqref{467} implies 
$$
V(x;\kappa+\ep) = l^{(\kappa+\ep)}(x), \qquad x \in [0, \beta_{\kappa+\ep}).
$$
By taking $\ep \to 0$, we get $V(x;\kappa) = l^{(\kappa)}(x)$ for $x \in [0, \limsup_{\ep \to 0} \beta_{\kappa+\ep})$. By continuity
$$
V(\cdot;\kappa) = l^{(\kappa)} \quad \text{for } x \in [0, \limsup_{\ep \to 0} \beta_{\kappa + \ep}].
 $$
 Note from Proposition~\ref{prop_44}, that $\beta_{\kappa} = \hat{\beta}_{\kappa}$, in which $\hat{\beta}_{\kappa}$ is given by \eqref{441}. Consequently,
 \begin{equation}\label{eq:beta-2}
 \limsup_{\ep \to 0}  \beta_{\kappa + \ep} \leq \beta_{\kappa}.
 \end{equation}
 From \eqref{eq:beta-1} and \eqref{eq:beta-2}, we obtain $\lim_{\ep \to 0} \beta_{\kappa + \ep}$ exists and equals $\beta_{\kappa}$.  Therefore, we have shown the continuity of $\kappa \mapsto \beta_{\kappa}$.
 
 \vspace{0.1in}
 
 \emph{Step 4: Continuity of $\kappa \mapsto V'(\cdot;\kappa)$.}
From \eqref{467}, we have $V'(x;\kappa+\ep)=(l^{(\kappa + \ep)})'(x)$ for $x\in[0,\beta_{\kappa+\ep}]$.  It follows from the continuity of $\kappa \mapsto (l^{(\kappa)})'(\cdot)$ obtained in Step 2 that
\begin{equation}\label{eq:V'-cont}
\lim_{\ep \to 0} V'(x;\kappa + \ep) = \lim_{\ep \to 0} (l^{(\kappa + \ep)})'(x) = (l^{(\kappa)})'(x) = V'(x;\kappa) \quad \text{for }x \in [0, \liminf_{\ep \to 0} \beta_{\kappa + \ep}].
\end{equation}
Recall $V'(\cdot;\kappa) \equiv 1$ on $[\beta_{\kappa}, \infty)$. From  the continuity of $\kappa \mapsto \beta_{\kappa}$, we now conclude, from \eqref{eq:V'-cont}, that $\kappa \mapsto V'(x; \kappa)$ is continuous, uniformly in $x \in \R_+$. 
\end{proof}

\footnotesize
\bibliographystyle{abbrv}
\bibliography{refs}
\end{document}